\documentclass[11pt,reqno]{amsart}
 \usepackage{textcomp,multicol,enumerate,amsmath,amssymb,amsthm,latexsym,dsfont}
 \usepackage{amsfonts,mathtools,wasysym}
 \usepackage{xcolor}
 
 \usepackage[T1]{fontenc}
 \usepackage{fourier}
 \usepackage{bbm}
 \usepackage{color}
 \usepackage{xcolor}
 \usepackage{tikz}
 \usepackage{fullpage} 
 \allowdisplaybreaks
 \usepackage{cleveref}
 \usepackage{bm}
 \usetikzlibrary{arrows}

 \usepackage[margin=2.5cm]{geometry}
 \newtheorem{theorem}{Theorem}[section]
 
 \newtheorem{lemma}[theorem]{Lemma}
  \newtheorem{claim}[theorem]{Claim}
 \newtheorem{proposition}[theorem]{Proposition}
 \newtheorem{conjecture}[theorem]{Conjecture}
\newtheorem{remark}[theorem]{Remark}
 \newtheorem{question}[theorem]{Question}
  \theoremstyle{definition}

\newcommand{\eps}{\varepsilon}

\newcommand{\prob}[1]{\mathbb{P}\left(#1 \right)} 

\newcommand{\expect}[1]{\mathbb{E}\left(#1 \right)} 
\newcommand{\condexpect}[2]{\mathbb{E}\left(#1 \;\middle|\; #2\right)} 

\newcommand{\order}[1]{\left|V\left(#1\right)\right|} 

\definecolor{darkgreen}{rgb}{0,.65,0}

\title{The emergence of a giant rainbow component}

\author{Oliver Cooley$^*$ , Tuan Anh Do$^{\dagger}$ , Joshua Erde$^{\dagger}$ , Michael Missethan$^{\dagger}$ }

\thanks{$^*$
Institute of Science and Technology Austria (ISTA),
Am Campus 1, 3400 Klosterneuburg, Austria, {\tt oliver.cooley@ist.ac.at}}

\thanks{$^{\dagger}$ 
 	Institute of Discrete Mathematics, 
 	Graz University of Technology, 
 	Steyrergasse 30,
 	8010 Graz,
 	Austria,  
 	{\tt \{do,erde,
 	missethan\}@math.tugraz.at}. Supported by Austrian Science Fund (FWF) : P36131, W1230}
 	
\begin{document}

\maketitle
\begin{abstract} 
The random coloured graph $G_c(n,p)$ is
obtained from the Erd\H{o}s-R\'{e}nyi binomial random graph $G(n,p)$ by assigning to each edge a colour from a set of $c$ colours independently and uniformly at random. It is not hard to see that, when $c = \Theta(n)$, the order of the largest rainbow tree in this model undergoes a phase transition at the critical point $p=\frac{1}{n}$. In this paper we determine the asymptotic order of the largest rainbow tree in the \emph{weakly sub-  and supercritical regimes}, when $p = \frac{1+\eps}{n}$ for some $\eps=\eps(n)$ which satisfies $\eps = o(1)$ and $|\eps|^3 n\to\infty$. In particular, we show that in both of these regimes with high probability the largest component of $G_c(n,p)$ contains an almost spanning rainbow tree. We also consider the order of the largest rainbow tree in the \emph{sparse regime}, when $p = \frac{d}{n}$ for some constant $d >1$. Here we show that the largest rainbow tree has linear order, and, moreover, for $d$ and $c$ sufficiently large, with high probability $G_c(n,p)$ even contains an almost spanning rainbow cycle.

\end{abstract}
\section{Introduction}
\subsection{Motivation}
In this paper we consider the following model of a random coloured graph: Given $c,n \in \mathbb{N}$ and $p \in (0,1)$, we let $G_c(n,p)$ be a randomly coloured graph generated by taking an Erd\H{o}s-R\'{e}nyi binomial random graph $G(n,p)$ and choosing the colour of each edge independently and uniformly from a set of $c$ colours. A natural question to consider in this model is the threshold for the appearance of certain \emph{rainbow substructures}, subgraphs all of whose edges receive a different colour. 

This model, or at least a graph process version of this model, was first studied by Frieze and McKay \cite{FM94}, who gave a hitting time result for the existence of a rainbow spanning tree when $c \geq n-1$, showing that it coincides with the hitting time for being connected and containing $n-1$ different colours, a clearly necessary condition. In particular, their result implies that there is a sharp threshold for the existence of a rainbow spanning tree in $G_{n-1}(n,p)$ at $p=\frac{2 \log n}{n}$ (see also \cite{B21}), since this is the threshold for $G_{n-1}(n,p)$ containing $n-1$ different colours.

There has also been much interest in the threshold for the existence of a rainbow Hamilton cycle in this model. It is a well-known result of Koml{\'o}s and Szemer{\'e}di that whp\footnote{Short for ``with high probability'', meaning with probability tending to one as $n \to \infty$. Here and throughout the paper, unless otherwise stated, our asymptotics will be taken as $n \to \infty$.} $G(n,p)$ contains a Hamilton cycle if $p = \frac{\log n + \log \log n + \omega(1)}{n}$. After some earlier work \cite{BF16,CF02,FL14}, the current best results regarding the existence of rainbow Hamilton cycles in $G_c(n,p)$ are due to Ferber and Krivelevich \cite{FK16}, who showed that in the optimum range of $p$, where $p = \frac{\log n + \log \log n + \omega(1)}{n}$, whp $G_c(n,p)$ contains a rainbow Hamilton cycle if $c= (1+\eps)n$ for some fixed $\eps > 0$, and Ferber \cite{F15}, who showed that with the optimum number of colours $c=n$, whp $G_n(n,p)$ contains a rainbow Hamilton cycle if $p=\frac{K\log n}{n}$ for an appropriately large $K$. This constant cannot be improved to $K=1$, since for $K<2$ we do not expect to see every colour in $G_n(n,p)$, but it would be interesting to know if the statement holds for any $K >2$, s with the result of Frieze and McKay \cite{FM94} for arbitrary rainbow trees mentioned above. The existence of more general rainbow spanning structures has also been considered in this model \cite{FK16,FNP16}.

There has also been some interest in the existence of rainbow Hamilton cycles in random colourings of other models of random graphs, for example in random regular graphs \cite{JW07}, random digraphs \cite{F15}, randomly perturbed graphs \cite{AH20}, random geometric graphs \cite{BBPP17} and random hypergraphs \cite{DEF18,FK16}.

We also note that some of the aforementioned results can be deduced as corollaries of an extension to the rainbow setting, proved by Bell, Frieze and Marbach~\cite{FM21}, of the recent breakthrough of Frankston, Kahn, Narayanan and Park \cite{FKNP19} on Talagrand's fractional version of the `Expectation-threshold' conjecture of Kahn and Kalai. This conjecture was recently proved by Park and Pham \cite{PP22}.

Whereas most of the previous results in
the $G_c(n,p)$ model
have focused on the behaviour around the connectivity/Hamiltonicity threshold, our focus will instead be on the behaviour of $G_c(n,p)$ with $c=\Theta(n)$ and with $p$ close to the critical point $\frac{1}{n}$.
It is well known that around this probability the structure of the underlying random graph $G(n,p)$ changes dramatically - when $p$ is significantly smaller than $\frac{1}{n}$, whp $G(n,p)$ will consist of many small components, whereas when $p$ is significantly larger than $\frac{1}{n}$, whp $G(n,p)$ will contain a unique \emph{giant} component. More precisely, we have the following.
\begin{theorem}[\cite{ER59}]\label{t:giantsparse}
Let $d>0$, let $p=\frac{d}{n}$ and let $\gamma = \gamma(d)$ be the survival probability of a Po$(d)$ branching process.
\begin{enumerate}[(a)]
    \item If $d < 1$, then whp all components in $G(n,p)$ have order $O(\log n)$;
   \item If $d>1$, then whp there is a unique \emph{giant} component in $G(n,p)$ of order $(\gamma + o(1))n$, and all other components have order $O(\log n)$. 
\end{enumerate}
\end{theorem}
\noindent Note, in particular, that as $d \to \infty$ we have $\gamma(d) \to 1$.

Whilst at first it may seem that the structure of $G(n,p)$ undergoes quite a sharp change in behaviour at this point, subsequent work, notably by Bollob\'{a}s \cite{B84} and {\L}uczak \cite{L90}, showed that in fact, if one chooses the correct parameterisation for $p$, this change can be seen to happen quite smoothly. In particular, {\L}uczak proved the following result.
\begin{theorem}[\cite{L90}]\label{t:giantLuczak}
Let $\eps=\eps(n)>0$ be such that $\eps^3 n \rightarrow \infty$ and $\eps = o(1)$, let $\gamma=\gamma(1+\eps)$ be the survival probability of a Po$(1+\eps)$ branching process,
and for each $i \in \mathbb{N}$ let $L_i$ be the $i$-th largest component in $G(n,p)$. 
\begin{enumerate}[(a)]
      \item\label{i:Lsub} If $p=\frac{1-\eps}{n}$, then whp $L_1$ is a tree of order $(1+o(1))\frac{2}{\eps^{2}}\log\left(\eps^3 n\right)$;
    \item\label{i:Lsup} If $p=\frac{1+\eps}{n}$, then whp  $|V(L_1)| = \left(\gamma + o(1)\right)n$
    and $|V(L_2)|=(1+o(1))\frac{2}{\eps^{2}}\log\left(\eps^3 n\right)$. 
\end{enumerate}
In particular,
\[
\left|V(L_1)\right| = \left(2\eps+O\left(\eps^2\right)\right) n \qquad \text{and} \qquad \left|V(L_2)\right| \leq n^{\frac{2}{3}}.
\]
\end{theorem}

\subsection{Main results}

We will be interested in the appearance of large rainbow structures in $G_c(n,p)$ for a similar range of $p$. However, there does not seem to be a natural definition of a rainbow component,
essentially because the property of being connected
is monotone \emph{increasing} with respect to the edges,
while the property of being rainbow is monotone \emph{decreasing}.
Thus, the vertex-maximal connected rainbow subgraphs may not partition the vertices,
and the non-maximal connected rainbow subgraphs may be extended in multiple
but incompatible ways.

Nevertheless,
in the case of uncoloured graphs, since every component has a spanning tree,
there is a clear equivalence between the order of the largest component and
the order of the largest tree. For this reason, in the coloured setting,
it is perhaps natural to consider the order of the largest rainbow tree
as a rainbow analogue of components.
Our main result concerns the \emph{weakly sub- and supercritical} regimes, where $p = \frac{1+\eps}{n}$, with $\eps=\eps(n)$ such that $|\eps|^3 n \rightarrow \infty$ and $\eps=o(1)$. In this case, by \Cref{t:giantLuczak}, clearly we cannot hope to find a rainbow tree of order larger than $(1+o(1))\frac{2}{\eps^{2}}\log\left(\eps^3 n\right)$ in the weakly subcritical regime and $(2+o(1)) \eps n$ in the weakly supercritical regime. We show that these trivial (given \Cref{t:giantLuczak}) upper bounds are indeed best possible.

\begin{theorem}\label{thm:giantrainbowtree}
Let $c=\Theta(n)$, let $\eps=\eps(n)>0$ be such that $\eps^3 n \to \infty$ and $\eps = o(1)$.  
\begin{enumerate}[(a)]
    \item\label{i:sub} If $p=\frac{1-\eps}{n}$, then whp the largest rainbow tree in $G_c(n,p)$ has order $(1+o(1))\frac{2}{\eps^{2}}\log\left(\eps^3 n\right)$;
    \item\label{i:sup} If $p=\frac{1+\eps}{n}$, then whp the largest rainbow tree in $G_c(n,p)$ has order $(2+o(1))\eps n$.
\end{enumerate}
\end{theorem}

For small constant $\eps>0$ we have $\gamma(\eps) = 2\eps + O\left(\eps^2\right)$, so by \Cref{t:giantsparse}, in $G_c\left(n, \frac{1+\eps}{n}\right)$ we cannot hope to find a rainbow tree of order significantly larger than $2\eps n$. While we cannot show such a tight bound, it is relatively easy to show, by comparison with a branching process, that whp $G_c(n,p)$ contains a rainbow tree of order $\Omega(\eps n)$.

\begin{theorem}\label{t:lineartreesparse}
Let $\alpha >0$, let $c=\alpha n$, let $\eps > 0$ be a sufficiently small constant and let $p= \frac{1+\eps}{n}$. Then whp $G_c(n,p)$ contains a rainbow tree of order at least $\left(\frac{\alpha}{\alpha +1}\eps - O\left( \eps^2\right)\right)n$.
\end{theorem}

Finally, in light of \Cref{t:giantLuczak} and since $\gamma(\eps)\xrightarrow{\eps \to \infty}1$,
we might hope that for sufficiently large $d$, whp $G_c(n,d/n)$ will contain a rainbow tree of \emph{almost} the optimal possible order  $\min\{n,c+1\}$.
In fact, for large enough $d$ we will see that whp $G_c\left(n,\frac{d}{n}\right)$ will contain even a rainbow cycle of almost optimal length, extending a result of Aigner-Horev and Hefetz \cite[Corollary 2.3]{AH20}, who proved the likely existence of a rainbow almost spanning path in $G_{(1+\gamma)n}\left(n,\frac{d}{n}\right)$ for constant $\gamma >0$ and $d$ sufficiently large.

\begin{theorem}\label{t:almostspanningcycle}
Let $\alpha,\delta >0$ and let $c=\alpha n$. Then there exists $d:=d(\delta)$ such that whp $G_c\left(n,\frac{d}{n}\right)$ contains a rainbow cycle of length at least $(1-\delta) \min\{n,c\}$.
\end{theorem}

Similar methods will also imply the existence of a linear length rainbow cycle for arbitrary $d>1$.

\subsection{Proof outline and key ideas}\label{sec:outline}
The overall strategy to prove \Cref{thm:giantrainbowtree}
is first to reveal the largest component of $G(n,p)$, and then to reveal
the colours of the edges, discarding some edges until we have at most
one of each colour. The trick is to choose which edges to discard and which to keep in such a
way that, while the largest component may split into smaller rainbow parts, one of these will cover almost all of the largest component.

In the weakly subcritical regime these choices are relatively simple - whp the largest component is a tree and we will see that if we delete a random edge from a random tree, one of the resulting components will likely cover almost all of the original vertex set. In particular, this effect is so pronounced that even if we deleted all the edges which share colours, we still expect there to be a tree in what remains which contains almost all the vertices of the largest component.

In the weakly supercritical regime we have to work a bit harder.
To decide which edges to discard and which to keep,
we partition the edges of the giant component into its $2$-core $C$
and the remaining forest $F$, rooted in $C$. Then when revealing colours,
we proceed as follows.
\begin{enumerate}
    \item \label{delete:coreforest} If a colour appears on an edge $e$ of $C$ and an edge $f$ of $F$,
    we delete $f$.
    \item \label{delete:forestthree} If a colour appears at least three times in $F$, we delete all the corresponding edges.
    \item \label{delete:foresttwo} If a colour appears exactly twice in $F$, we delete the ``better'' edge, i.e.\ the one whose deletion will remove fewer vertices.
    \item \label{delete:core} If a colour appears at least twice in $C$, we delete all the corresponding edges.
\end{enumerate}
We note that more than one case can occur for each edge, but since we only need an upper
bound on the number of vertices disconnected from the giant component in this manner,
any potential multiple counting will not be a problem.

Clearly~\eqref{delete:forestthree} and~\eqref{delete:core} are slightly
crude ways of proceeding, since we could keep one of these edges in each case.
However, it turns out that there are few enough of these edges that
they make little difference.

Similarly,~\eqref{delete:coreforest} might naively 
seem an odd way to proceed: In $C$ we could delete an edge without
decreasing the order of the giant if it lies in a cycle, but deleting
the edge of $F$ certainly causes some loss. However, our slightly
counterintuitive strategy turns out to be a better one --
the heuristic explanation is that it is important
to protect the core and avoid it splitting into multiple smaller components,
even at the cost of losing some of the surrounding forest.

Finally,~\eqref{delete:foresttwo} is the most delicate of the conditions
to analyse. If we were to delete one of the two edges arbitrarily,
the expected number of edges we lose from the
giant component is $\Theta(1/\eps)$.
We also expect this to occur $\Theta(\eps^2 n)$ times,
leading to a total loss of $\Theta(\eps n)$ -- the same as the
order of the giant component, and therefore too much for our goal.
However, it turns out that by choosing the better edge,
the expected loss drops to $o(1/\eps)$,
(see \Cref{prop:expminbnforest}) which is precisely the improvement
we need.

To show \Cref{t:lineartreesparse,t:almostspanningcycle} we analyse rainbow versions of the breadth- and depth-first search algorithms to show the existence of large rainbow trees and paths, respectively. Given the likely existence of a large rainbow path, a standard sprinkling argument proves the likely existence of a rainbow cycle of roughly the same length.

\subsection{Outline of the paper}
In \Cref{s:prelim} we collect some preliminary results which will be useful later in the paper and in particular consider the structure of random forests. In \Cref{s:weakly} we prove \Cref{thm:giantrainbowtree} and in \Cref{s:sparse} we discuss the sparse regime and prove \Cref{t:lineartreesparse,t:almostspanningcycle}. Finally in \Cref{s:discussion} we discuss some open problems and directions for future research.

\section{Preliminaries}\label{s:prelim}
\subsection{Asymptotics}
Given two functions $f,g: \mathbb{N} \to \mathbb{R}$, in a slight abuse of notation we will write statements of the form ``If $\eps >0$ is sufficiently small, then $f = O(\eps  g)$'' to mean that there exist constants $\eps_0,C>0$ such that $f(n) \leq C \eps g(n)$ for all $\eps<\eps_0$ and for all $n$. See for example the statement of \Cref{t:lineartreesparse}.

We will ignore floors and ceilings whenever these do not significantly affect the argument.

\subsection{The configuration model}
Given a degree sequence ${\bf d} \in \mathbb{N}^s$, the \emph{configuration model} constructs a random multigraph $G^*({\bf d})$ in the following manner: Let $\mathcal{W}({\bf d})=\{W_1,\ldots, W_s\}$ where $(|W_1|,|W_2|,\ldots, |W_s|) = {\bf d}$ and the sets $W_i$ are pairwise disjoint. We call the $W_i$ \emph{cells} and the elements of the $W_i$ \emph{half-edges}. A \emph{configuration} is a partition $M$ of $W := \bigcup_{i \in [s]} W_i$ into pairs, which we think of as a perfect matching on the set of half-edges. 

The (multi-)graph $G^*({\bf d})$ is formed by choosing a configuration $M$ uniformly at random and taking the \mbox{(multi-)}graph $G(\mathcal{W},M)$ whose vertex set is $[s]$ and where we have an edge between $i$ and $j$ for each partition class of $M$
whose elements lie in $W_i$ and $W_j$.

We note that if we sequentially choose an arbitrary unmatched half-edge and choose a partner for it uniformly at random from the set of unmatched half edges, then the configuration $M$ that we obtain in this manner is distributed uniformly at random. In this way, we can think of the edges in $G^*(\bm{d})$ as being exposed sequentially.

\subsection{Chernoff bound}

We will frequently use the following form of the Chernoff bound, which follows from
e.g.~\cite[Theorem 2.1]{JansonLuczakRucinskiBook}.
\begin{lemma}[Chernoff bound]\label{lem:chernoff}
	If $X\sim \mathrm{Bin}(N,q)$, then for any $\gamma> 0$, setting $\mu:=Nq$ we have
\[
	\prob{|X-\mu|\geq \gamma \mu}\leq 2\cdot \exp\left(-\frac{\gamma^2 }{2\left(1+\frac{\gamma}{3}\right)}\cdot \mu\right).
\]
\end{lemma}

\subsection{Random forests}

For given $m,t \in \mathbb{N}$ let $\mathcal{F}(m,t)$ be the class of all forests on vertex set $[m]$ having $t$ trees such that the vertices $1, \ldots, t$ all lie in different trees. We denote by $F(m,t)$ a forest chosen uniformly at random from the class $\mathcal{F}(m,t)$ and call the vertices $1, \ldots, t$ the \emph{roots} of the trees. Throughout this section, all asymptotics are taken as $m\to \infty$ and $t=t(m)$ is a function in $m$.

We will often use the well-known fact, see e.g., \cite{R59}, that
\begin{equation}\label{eq:numberofforests}
 \left|\mathcal{F}(m,t)\right|=tm^{m-t-1}.   
\end{equation}

Firstly, we will need a fact about the random forest that arises when we delete a random edge in a random tree. In this particular case, we will see that we expect one of the two components of this forest to be significantly smaller than the other.
\begin{proposition}\label{prop:unifedgedelete}
Let $e$ be a uniformly chosen random edge from $F(m,1)$ and let $T_1$ and $T_2$ be the two components of $F(m,1)-e$. Then we have
\[
\expect{\min\left\{|V(T_1)|,|V(T_2)|\right\}}=O\left(\sqrt{m}\right).
\]
\end{proposition}
\begin{proof}
We can construct each realisation of $F(m,1)$ and $e$ as follows: First we pick $u \neq v\in[m]$ and $k\in[m-1]$. Then we choose disjoint subsets $S_1$ and $S_2$ such that $S_1\cup S_2=[m]\setminus\left\{u,v\right\}$ and $|S_1|=k-1$. Next, we pick trees $T_1$ and $T_2$ on vertex sets $S_1$ and $S_2$, respectively. Finally, we obtain $F(m,1)$ by connecting $T_1$ and $T_2$ by the edge $e=uv$. Hence, we have 
\begin{align}\label{eq:minTreeComponent}
   \expect{\min\big\{|V(T_1)|,|V(T_2)|\big\}}\leq 2\sum_{k=1}^{m/2} k \cdot \prob{|V(T_1)|=k}= 2 \sum_{k=1}^{m/2} k \cdot \frac{\binom{m}{2}\binom{m-2}{k-1}k^{k-2}(m-k)^{m-k-2}}{m^{m-2}(m-1)},
\end{align} 
where the numerator on the right-hand side counts the number of different choices in the above construction that satisfy $|V(T_1)|=k$ and the denominator the number of different realisations of $(F(m,1),e)$. Using Stirling's formula in \eqref{eq:minTreeComponent} yields 
  \begin{align*}
   \expect{\min\big\{|V(T_1)|,|V(T_2)|\big\}}\leq\Theta(1)\sum_{k=1}^{m/2} \frac{\left(m^{m+1/2}/e^m\right) k^k(m-k)^{m-k-1}}{\left(k^{k+1/2}/e^{k}\right)\left((m-k)^{m-k+1/2}/e^{m-k}\right)m^{m-1}}
   =\Theta(1)\sum_{k=1}^{m/2}\frac{1}{k^{1/2}}
   =
   \Theta\left(\sqrt{m}\right),
\end{align*}
as required.
\end{proof}

The following two observations, both trivial from the definition of $F(m,t)$, will be useful later.

\begin{remark}\label{rem:expectedtreesize}
Let $r\in[t]$, let $v\in[m]$ and let $T_r$ be the component of $F(m,t)$ containing $r$. Then
\begin{enumerate}[(a)]
    \item \label{r:expectedtreesize} $\mathbb{E}(|V(T_r)|) = \frac{m}{t}$;
    \item \label{r:conditionalexpectedtreesize} $\condexpect{\order{T_r}}{v\notin V(T_r)}\leq \frac{m}{t}$.
\end{enumerate}
\end{remark}

Whenever we write $vw$ for an edge in $H\in\mathcal{F}(m,t)$, we tacitly assume that $v$ is closer than $w$ to the root of the tree containing $vw$. We note that then $w\notin [t]$. Given an edge $e=vw$ in the forest $H$, we define the bridge number $B_e$ to be the order of the component of $H-e$ containing $w$.

\begin{claim}\label{claim:expbridgenumberforest}
Let $e$ be a uniformly randomly chosen edge from $F(m,t)$. Then $\mathbb{E}(B_e) \leq \frac{m}{t+1}$.
\end{claim}

\begin{proof}
We will determine $\condexpect{B_e}{e=vw}$ for $v\in[m]$ and $w\in[m]\setminus[t]$ such that $v\neq w$. To simplify notation, we assume wlog that $w=t+1$. Then each forest $H\in\mathcal{F}(m,t)$ containing the edge $vw$ can be obtained by adding the edge $vw$ to a forest $H'\in\mathcal{F}(m,t+1)$ in which $v$ and $w$ are not in the same component. Then the bridge number $B_{vw}$ in $H$ equals the order of the tree of $H'$ containing $w$. Hence, letting $T_w$ be the component of $F(m,t+1)$ containing $w$ we obtain
\begin{align*}
\condexpect{B_e}{e=vw}=\condexpect{\order{T_w}}{v\notin V(T_w)}\leq \frac{m}{t+1},
\end{align*}
where the last inequality follows by \Cref{rem:expectedtreesize}\eqref{r:conditionalexpectedtreesize}. As this is true for all pairs $vw$ that can form an edge in $F(m,t)$, the statement follows.
\end{proof}

\Cref{rem:expectedtreesize}\eqref{r:expectedtreesize} identifies the expected order of a tree in $F(m,t)$ as $m/t$.
However, if we have two trees to choose from, perhaps surprisingly,
it is likely that one will be significantly smaller.
This is critical to point~\eqref{delete:foresttwo} of the proof strategy
described in \Cref{sec:outline},
and is quantified in the following proposition.

\begin{proposition}\label{prop:expmincomponentsforest}
Let $t=o(m)$ be such that $t=\omega(1)$, let $r\neq s$ be two distinct vertices from $[t]$ and let $T_r$ and $T_s$ be the components of $F(m,t)$ containing $r$ and $s$, respectively. Then we have
\[
\expect{\min\big\{\order{T_r},\order{T_s}\big\}} = o(m/t).
\]
\end{proposition}

We note that \Cref{prop:unifedgedelete} may be considered a
(slightly more precise) analogue of \Cref{prop:expmincomponentsforest}
in the case when $t=2$.

\begin{proof}
When we delete $T_r$ from $F(m,t)$ we obtain a random forest with $m-\order{T_r}$ vertices and $t-1$ trees. Hence, for each fixed constant $k\in\mathbb{N}$ we have
\begin{align*}
    \prob{\order{T_r}=k}
    & \stackrel{\phantom{\eqref{eq:numberofforests}}}{=} \binom{m-t}{k-1}k^{k-2}\frac{|\mathcal{F}(m-k,t-1)|}{|\mathcal{F}(m,t)|}\\
    & \stackrel{\eqref{eq:numberofforests}}{=} \binom{m-t}{k-1}k^{k-2}\frac{(t-1)(m-k)^{m-k-t}}{tm^{m-t-1}}\\
    & \stackrel{\phantom{\eqref{eq:numberofforests}}}{=} (1+o(1))\frac{m^{k-1}}{(k-1)!}\frac{k^{k-1}}{k}\frac{(1-k/m)^{m-k-t}}{m^{k-1}}\\
    & \stackrel{\phantom{\eqref{eq:numberofforests}}}{=} (1+o(1))\frac{e^{-k}k^{k-1}}{k!}.
\end{align*}
The term $\frac{e^{-k}k^{k-1}}{k!}$ is precisely the probability mass function of the
Borel distribution with parameter $1$, and so we have $\sum_{k\in\mathbb{N}}\left(e^{-k}k^{k-1}/k!\right)=1$, which implies that
$\lim_{k\to \infty} \prob{|V(T_r)|\ge k} = o(1)$, or in other words
\begin{align}\label{eq:forest2}
 \text{whp} \quad \order{T_r}\leq h
\end{align}
for any function $h=h(m)=\omega(1)$. We set $h=\sqrt{m/t}$ and consider the random variable $Z$ defined by
\[
Z:=
\begin{cases}
\order{T_r} & \mbox{if }\order{T_r}\leq h,\\
\order{T_s} & \mbox{otherwise.}
\end{cases}
\]
We note that
\begin{align}\label{eq:forest3}
\min\big\{\order{T_r},\order{T_s}\big\}\leq Z.
\end{align}
Furthermore, we have
\begin{align*}
    \expect{Z}&=\prob{\order{T_r}\leq h}\condexpect{Z}{\order{T_r}\leq h}+\sum_{i>h}\prob{\order{T_r}=i}\condexpect{Z}{\order{T_r}=i}
    \\
    &\leq 
    \prob{\order{T_r}\leq h}\cdot h+\prob{\order{T_r}> h}\frac{m}{t-1},
\end{align*}
where in the last inequality we used that $\condexpect{Z}{\order{T_r}=i}=\condexpect{\order{T_s}}{\order{T_r}=i}=(m-i)/(t-1)\leq m/(t-1)$ for each $i>h$ (since $r,s \in [t]$ and therefore $s \notin T_r$). This yields
\[
\expect{Z} \stackrel{\scriptsize \eqref{eq:forest2}}{\leq} 1\cdot h+o(1)\frac{m}{t-1}=o\left(\frac{m}{t}\right)
\]
Together with \eqref{eq:forest3} this shows the statement.
\end{proof}

We now aim to apply this proposition to show that, given the choice of two edges
in a random forest $F(m,t)$, the expectation of the smaller bridge number is
$o(m/t)$. The main difficulty is to handle the possibility
that the two edges might lie in the same tree.

\begin{proposition}\label{prop:expminbnforest}
Let $t=o(m)$ be such that $t=\omega(1)$
and let $e,e'$ be two edges of $F(m,t)$
selected uniformly at random and independently
subject to the condition that $e\neq e'$. Then we have
\[\mathbb{E}\left(\min\big\{B_e,B_{e'}\big\}\right) = o(m/t).\]
\end{proposition}

\begin{proof}
We observe that for each vertex $w\in[m]\setminus[t]$ there is a unique edge $e(w)=vw$ in $F(m,t)$ in which $w$ is further away than $v$ from the root of the tree component containing $vw$. Hence, due to symmetry we have 
\begin{align}\label{eq:wlog}
 \expect{\min\big\{B_e,B_{e'}\big\}}=\expect{\min\left\{B_{e(t+1)},B_{e(t+2)}\right\}}.
\end{align}

To determine the right-hand side in \eqref{eq:wlog}, we consider the following construction. We pick, uniformly at random and independently of each other,
a forest $F=F(m,t+2)$ from $\mathcal{F}(m,t+2)$,
a vertex $u\in[m]$, and another vertex $v\in[m]$.
Then we let $F'$ be the (multi-)graph obtained from $F$ by adding an edge between $u$ and $t+1$ and an edge between $v$ and $t+2$ (see \Cref{fig:minBridgeNumber}).

\newcommand{\lw}{1pt}
\newcommand{\vs}{2.8pt}
\newcommand{\shiftA}{-0.15}
\newcommand{\shiftB}{-0.19}
\newcommand{\shiftC}{0.1}
\newcommand{\shiftD}{0.02}
\begin{figure}[h]
    \centering
    \begin{tikzpicture}[line cap=round,line join=round,>=triangle 45,x=1cm,y=1cm]
\draw [line width=\lw] (0,0)-- (-0.6,0.6);
\draw [line width=\lw] (0,0)-- (0.6,0.8);
\draw [line width=\lw] (0,0)-- (0,0.8);
\draw [line width=\lw] (0,0.8)-- (-0.5,1.3);
\draw [line width=\lw] (0.6,0.8)-- (0.7,1.4);
\draw [line width=\lw] (0.6,0.8)-- (1.13,1.2);
\draw [line width=\lw] (0,2.2)-- (0,2.8);
\draw [line width=\lw] (0,2.8)-- (-0.3,3.3);
\draw [line width=\lw] (0,2.8)-- (0.3,3.3);
\draw [line width=2pt,dash pattern=on 0.3pt off 3pt] (0,0.8)-- (0,2.2);
\draw [rotate around={89.78127510348422:(0.0025268716333589993,2.8619199687825794)},line width=1.8pt] (0.0025268716333589993,2.8619199687825794) ellipse (0.8950219865802cm and 0.6024283578223942cm);
\draw [line width=\lw] (2.7+\shiftA,0)-- (2.7+\shiftA,0.6);
\draw [line width=\lw] (2.7+\shiftA,0.6)-- (2.1+\shiftA,1.2);
\draw [line width=\lw] (2.7+\shiftA,0.6)-- (3.2+\shiftA,1);
\draw [line width=\lw] (3.2+\shiftA,1)-- (3.6+\shiftA,1.5);
\draw [line width=\lw] (3.2+\shiftA,1)-- (3.8+\shiftA,1.1);
\draw [line width=\lw] (2.6+\shiftA,2.2)-- (2.7+\shiftA,3);
\draw [line width=\lw] (2.6+\shiftA,2.2)-- (2.2+\shiftA,2.9);
\draw [line width=\lw] (2.6+\shiftA,2.2)-- (3.2+\shiftA,3);
\draw [line width=\lw] (3.2+\shiftA,3)-- (3.3+\shiftA,3.5);
\draw [line width=\lw] (3.2+\shiftA,3)-- (3.6+\shiftA,3.4);
\draw [line width=\lw] (2.2+\shiftA,2.9)-- (1.8+\shiftA,3.6);
\draw [line width=\lw] (2.2+\shiftA,2.9)-- (2.2+\shiftA,3.6);
\draw [line width=\lw] (2.2+\shiftA,2.9)-- (2.6+\shiftA,3.5);
\draw [line width=2pt,dash pattern=on 0.3pt off 3pt] (2.7+\shiftA,0.6)-- (2.6+\shiftA,2.2);
\draw [rotate around={0:(2.6528248723773826+\shiftA,3.1)},line width=0.6pt,dotted] (2.6528248723773826+\shiftA,3.1) ellipse (1.2326334477367022cm and 1.1860623857953614cm);
\draw (-0.2,-0.22) node[anchor=north west] {neither $A$ nor $B$ hold};
\draw [line width=\lw] (6.1+\shiftB,0)-- (5.7+\shiftB,0.4);
\draw [line width=\lw] (6.1+\shiftB,0)-- (6.1+\shiftB,0.6);
\draw [line width=\lw] (6.1+\shiftB,0)-- (6.6+\shiftB,0.3);
\draw [line width=\lw] (5.7+\shiftB,0.4)-- (5.4+\shiftB,0.9);
\draw [line width=\lw] (6.6+\shiftB,0.3)-- (7+\shiftB,0.7);
\draw [line width=\lw] (7.9+\shiftB,0)-- (7.5+\shiftB,0.5);
\draw [line width=\lw] (7.9+\shiftB,0)-- (8.1+\shiftB,0.5);
\draw [line width=\lw] (8.1+\shiftB,0.5)-- (7.7+\shiftB,0.9);
\draw [line width=\lw] (8.1+\shiftB,0.5)-- (8.3+\shiftB,1);
\draw [line width=\lw] (8.3+\shiftB,1)-- (8.1+\shiftB,1.5);
\draw [line width=\lw] (6.2+\shiftB,1.3)-- (5.9+\shiftB,1.8);
\draw [line width=\lw] (6.2+\shiftB,1.3)-- (6.25+\shiftB,1.9);
\draw [line width=\lw] (6.2+\shiftB,1.3)-- (6.6+\shiftB,1.8);
\draw [line width=\lw] (6.6+\shiftB,1.8)-- (6.6+\shiftB,2.3);
\draw [line width=\lw] (6.6+\shiftB,1.8)-- (7+\shiftB,2.1);
\draw [line width=\lw] (7+\shiftB,2.1)-- (7.2+\shiftB,2.6);
\draw [line width=\lw] (7+\shiftB,2.1)-- (7.5+\shiftB,2.2);
\draw [line width=\lw] (6.6+\shiftB,2.3)-- (6.7+\shiftB,2.7);
\draw [line width=\lw] (5.9+\shiftB,2.7)-- (6+\shiftB,3.3);
\draw [line width=\lw] (5.9+\shiftB,2.7)-- (5.5+\shiftB,3.1);
\draw [line width=\lw] (6+\shiftB,3.3)-- (5.8+\shiftB,3.7);
\draw [line width=\lw] (6+\shiftB,3.3)-- (6.3+\shiftB,3.7);
\draw [line width=2pt,dash pattern=on 0.3pt off 3pt] (6.6+\shiftB,0.3)-- (6.2+\shiftB,1.3);
\draw [line width=2pt,dash pattern=on 0.3pt off 3pt](5.9+\shiftB,1.8)-- (5.9+\shiftB,2.7);
\draw [rotate around={-66.57130719125621:(5.9015771247123405+\shiftB,3.273436773025351)},line width=1.8pt] (5.9015771247123405+\shiftB,3.273436773025351) ellipse (0.78cm and 0.78cm);
\draw [rotate around={-61.27386933469538:(6.203999095033547+\shiftB,2.675353598209841)},line width=0.6pt,dotted] (6.203999095033547+\shiftB,2.675353598209841) ellipse (1.670928757124346cm and 1.479270662090687cm);
\draw (5.56,-0.22) node[anchor=north west] {\parbox{1.645638780672247 cm}{ $A$ holds}};
\draw [line width=\lw] (10.1+\shiftC,0)-- (9.7+\shiftC,0.3);
\draw [line width=\lw] (10.1+\shiftC,0)-- (10.1+\shiftC,0.6);
\draw [line width=\lw] (10.1+\shiftC,0)-- (10.5+\shiftC,0.2);
\draw [line width=\lw] (10.5+\shiftC,0.2)-- (10.8+\shiftC,0.7);
\draw [line width=\lw] (10.5+\shiftC,0.2)-- (11+\shiftC,0.3);
\draw [line width=\lw] (10.2+\shiftC,1.1)-- (9.8+\shiftC,1.5);
\draw [line width=\lw] (10.2+\shiftC,1.1)-- (10.1+\shiftC,1.8);
\draw [line width=\lw] (10.2+\shiftC,1.1)-- (10.5+\shiftC,1.7);
\draw [line width=\lw] (10.5+\shiftC,1.7)-- (10.8+\shiftC,2.1);
\draw [line width=\lw] (10.8+\shiftC,2.1)-- (11.4+\shiftC,2.3);
\draw [line width=\lw] (10.8+\shiftC,2.1)-- (11.2+\shiftC,2.6);
\draw [line width=\lw] (10.8+\shiftC,2.1)-- (10.8+\shiftC,2.6);
\draw [line width=2pt,dash pattern=on 0.3pt off 3pt] (10.8+\shiftC,2.6)-- (9.8+\shiftC,2.9);
\draw [line width=\lw] (9.8+\shiftC,2.9)-- (9.4+\shiftC,3.2);
\draw [line width=\lw] (9.8+\shiftC,2.9)-- (9.7+\shiftC,3.4);
\draw [line width=\lw] (9.8+\shiftC,2.9)-- (10.1+\shiftC,3.3);
\draw [line width=\lw] (9.8+\shiftC,2.9)-- (10.3+\shiftC,3);
\draw [line width=\lw] (10.3+\shiftC,3)-- (10.5+\shiftC,3.4);
\draw [line width=\lw] (10.1+\shiftC,3.3)-- (9.9+\shiftC,3.8);
\draw [line width=\lw] (10.1+\shiftC,3.3)-- (10.3+\shiftC,3.7);
\draw [line width=\lw] (11.8+\shiftC,0)-- (11.8+\shiftC,0.5);
\draw [line width=\lw] (11.8+\shiftC,0.5)-- (11.6+\shiftC,0.9);
\draw [line width=\lw] (11.8+\shiftC,0.5)-- (12+\shiftC,0.9);
\draw [line width=2pt,dash pattern=on 0.3pt off 3pt](10.5+\shiftC,0.2)-- (10.2+\shiftC,1.1);
\draw [rotate around={-2.121096396661424:(9.93+\shiftC,3.27)},line width=1.8pt] (9.93+\shiftC,3.27) ellipse (0.8cm and 0.74cm);
\draw [rotate around={78.36636600105876:(10.325101290647417+\shiftC,2.6622499346307817)},line width=0.6pt,dotted] (10.325101290647417+\shiftC,2.6622499346307817) ellipse (1.7745293135233593cm and 1.5235916015923434cm);
\draw (9.83,-0.22) node[anchor=north west] {\parbox{1.645638780672247 cm}{$B$ holds}};
\draw [line width=\lw] (13.7+\shiftD,0)-- (13.9+\shiftD,0.4);
\draw [line width=\lw] (13.9+\shiftD,0.4)-- (13.9+\shiftD,0.9);
\draw [line width=\lw] (13.7+\shiftD,0)-- (14.5+\shiftD,0.4);
\draw [line width=\lw] (13.9+\shiftD,0.4)-- (14.3+\shiftD,0.7);
\draw [line width=\lw] (14.5+\shiftD,0.4)-- (14.6+\shiftD,1);
\draw [line width=\lw] (14.5+\shiftD,0.4)-- (14.9+\shiftD,0.8);
\draw [line width=\lw] (14.5+\shiftD,0.4)-- (15+\shiftD,0.4);
\draw [line width=\lw] (14.9+\shiftD,0.8)-- (14.9+\shiftD,1.2);
\draw [line width=\lw] (14.9+\shiftD,0.8)-- (15.3+\shiftD,1);
\draw [line width=2pt,dash pattern=on 0.3pt off 3pt] (14.6+\shiftD,1)-- (14+\shiftD,1.6);
\draw [line width=\lw] (14+\shiftD,1.6)-- (13.8+\shiftD,2.1);
\draw [line width=\lw] (14+\shiftD,1.6)-- (14.2+\shiftD,2.2);
\draw [line width=\lw] (14+\shiftD,1.6)-- (14.5+\shiftD,1.9);
\draw [line width=\lw] (14.2+\shiftD,2.2)-- (14.1+\shiftD,2.7);
\draw [line width=\lw] (14.2+\shiftD,2.2)-- (14.4+\shiftD,2.7);
\draw [line width=\lw] (14.2+\shiftD,2.2)-- (14.6+\shiftD,2.5);
\draw [line width=\lw] (13.8+\shiftD,2.1)-- (13.2+\shiftD,2.4);
\draw [line width=\lw] (13.2+\shiftD,2.4)-- (13.3+\shiftD,2.9);
\draw [line width=\lw] (13.8+\shiftD,2.1)-- (13.7+\shiftD,2.7);
\draw [line width=2pt,dash pattern=on 0.3pt off 3pt] (13.2+\shiftD,2.4)-- (13.7+\shiftD,0);
\draw [line width=\lw] (14+\shiftD,3.2)-- (13.4+\shiftD,3.5);
\draw [line width=\lw] (14+\shiftD,3.2)-- (14+\shiftD,3.7);
\draw [line width=\lw] (14+\shiftD,3.7)-- (13.6+\shiftD,3.9);
\draw [line width=\lw] (14.6+\shiftD,3.3)-- (14.6+\shiftD,3.9);
\draw (13.1,-0.22) node[anchor=north west] {$A$ and $B$ hold};
\draw [line width=0.5pt,dash pattern=on 1pt off 2pt] (4.13,4.4)-- (4.13,-0.8);
\draw [line width=0.5pt,dash pattern=on 1pt off 2pt] (8.52,4.4)-- (8.52,-0.8);
\draw [line width=0.5pt,dash pattern=on 1pt off 2pt] (12.6,4.4)-- (12.6,-0.8);
\begin{scriptsize}
\draw [fill=black] (-0.1,-0.1) rectangle (0.1, 0.1);
\draw [fill=black] (-0.6,0.6) circle (\vs);
\draw [fill=black] (0,0.8) circle (\vs);
\draw [color=black] (0.19,0.8) node {\tiny $u$};
\draw [fill=black] (0.6,0.8) circle (\vs);
\draw [fill=black] (0.7,1.4) circle (\vs);
\draw [fill=black] (1.13,1.2) circle (\vs);
\draw [fill=black] (-0.5,1.3) circle (\vs);
\draw [fill=black] (-0.1,2.1) rectangle (0.1,2.3);
\draw [color=black] (0.24,2.4) node {\tiny $t+1$};
\draw [fill=black] (0,2.8) circle (\vs);
\draw [fill=black] (-0.3,3.3) circle (\vs);
\draw [fill=black] (0.3,3.3) circle (\vs);
\draw [fill=black] (2.6+\shiftA,-0.1) rectangle (2.8+\shiftA,0.1);
\draw [fill=black] (2.7+\shiftA,0.6) circle (\vs);
\draw [color=black] (2.89+\shiftA,0.56) node {\tiny $v$};
\draw [fill=black] (2.1+\shiftA,1.2) circle (\vs);
\draw [fill=black] (3.2+\shiftA,1) circle (\vs);
\draw [fill=black] (3.6+\shiftA,1.5) circle (\vs);
\draw [fill=black] (3.8+\shiftA,1.1) circle (\vs);
\draw [fill=black] (2.5+\shiftA,2.1) rectangle (2.7+\shiftA,2.3);
\draw [color=black] (2.92+\shiftA,2.22) node {\tiny $t+2$};
\draw [fill=black] (2.2+\shiftA,2.9) circle (\vs);
\draw [fill=black] (2.7+\shiftA,3) circle (\vs);
\draw [fill=black] (3.2+\shiftA,3) circle (\vs);
\draw [fill=black] (3.3+\shiftA,3.5) circle (\vs);
\draw [fill=black] (3.6+\shiftA,3.4) circle (\vs);
\draw [fill=black] (2.6+\shiftA,3.5) circle (\vs);
\draw [fill=black] (1.8+\shiftA,3.6) circle (\vs);
\draw [fill=black] (2.2+\shiftA,3.6) circle (\vs);
\draw [fill=black] (6+\shiftB,-0.1) rectangle (6.2+\shiftB, 0.1);
\draw [fill=black] (6.1+\shiftB,0.6) circle (\vs);
\draw [fill=black] (5.7+\shiftB,0.4) circle (\vs);
\draw [fill=black] (6.6+\shiftB,0.3) circle (\vs);
\draw [color=black] (6.79+\shiftB,0.3) node {\tiny $v$};
\draw [fill=black] (5.4+\shiftB,0.9) circle (\vs);
\draw [fill=black] (7+\shiftB,0.7) circle (\vs);
\draw [fill=black] (6.1+\shiftB,1.2) rectangle (6.3+\shiftB,1.4);
\draw [color=black] (6.55+\shiftB,1.3) node {\tiny $t+2$};
\draw [fill=black] (5.9+\shiftB,1.8) circle (\vs);
\draw [color=black] (5.68+\shiftB,1.8) node {\tiny $u$};
\draw [fill=black] (6.25+\shiftB,1.9) circle (\vs);
\draw [fill=black] (6.6+\shiftB,1.8) circle (\vs);
\draw [fill=black] (6.6+\shiftB,2.3) circle (\vs);
\draw [fill=black] (7+\shiftB,2.1) circle (\vs);
\draw [fill=black] (7.2+\shiftB,2.6) circle (\vs);
\draw [fill=black] (7.5+\shiftB,2.2) circle (\vs);
\draw [fill=black] (6.7+\shiftB,2.7) circle (\vs);
\draw [fill=black] (5.8+\shiftB,2.6) rectangle (6+\shiftB,2.8);
\draw [color=black] (6.2+\shiftB,2.89) node {\tiny $t+1$};
\draw [fill=black] (5.5+\shiftB,3.1) circle (\vs);
\draw [fill=black] (6+\shiftB,3.3) circle (\vs);
\draw [fill=black] (6.3+\shiftB,3.7) circle (\vs);
\draw [fill=black] (5.8+\shiftB,3.7) circle (\vs);
\draw [fill=black] (7.8+\shiftB,-0.1) rectangle (8+\shiftB,0.1);
\draw [fill=black] (7.5+\shiftB,0.5) circle (\vs);
\draw [fill=black] (8.1+\shiftB,0.5) circle (\vs);
\draw [fill=black] (7.7+\shiftB,0.9) circle (\vs);
\draw [fill=black] (8.3+\shiftB,1) circle (\vs);
\draw [fill=black] (8.1+\shiftB,1.5) circle (\vs);
\draw [fill=black] (10+\shiftC,-0.1) rectangle (10.2+\shiftC,0.1);
\draw [fill=black] (9.7+\shiftC,0.3) circle (\vs);
\draw [fill=black] (10.1+\shiftC,0.6) circle (\vs);
\draw [fill=black] (10.5+\shiftC,0.2) circle (\vs);
\draw [color=black] (10.66+\shiftC,0.07) node {\tiny $u$};
\draw [fill=black] (10.8+\shiftC,0.7) circle (\vs);
\draw [fill=black] (11+\shiftC,0.3) circle (\vs);
\draw [fill=black] (10.1+\shiftC,1) rectangle (10.3+\shiftC,1.2);
\draw [color=black] (10.52+\shiftC,1.12) node {\tiny $t+1$};
\draw [fill=black] (9.8+\shiftC,1.5) circle (\vs);
\draw [fill=black] (10.1+\shiftC,1.8) circle (\vs);
\draw [fill=black] (10.5+\shiftC,1.7) circle (\vs);
\draw [fill=black] (10.8+\shiftC,2.6) circle (\vs);
\draw [color=black] (10.59+\shiftC,2.5) node {\tiny $v$};
\draw [fill=black] (10.8+\shiftC,2.1) circle (\vs);
\draw [fill=black] (11.2+\shiftC,2.6) circle (\vs);
\draw [fill=black] (11.4+\shiftC,2.3) circle (\vs);
\draw [fill=black] (9.7+\shiftC,2.8) rectangle (9.9+\shiftC,3);
\draw [color=black] (9.82+\shiftC,2.71) node {\tiny $t+2$};
\draw [fill=black] (9.4+\shiftC,3.2) circle (\vs);
\draw [fill=black] (9.7+\shiftC,3.4) circle (\vs);
\draw [fill=black] (10.1+\shiftC,3.3) circle (\vs);
\draw [fill=black] (10.3+\shiftC,3) circle (\vs);
\draw [fill=black] (9.9+\shiftC,3.8) circle (\vs);
\draw [fill=black] (10.3+\shiftC,3.7) circle (\vs);
\draw [fill=black] (10.5+\shiftC,3.4) circle (\vs);
\draw [fill=black] (11.7+\shiftC,-0.1) rectangle (11.9+\shiftC,0.1);
\draw [fill=black] (11.8+\shiftC,0.5) circle (\vs);
\draw [fill=black] (11.6+\shiftC,0.9) circle (\vs);
\draw [fill=black] (12+\shiftC,0.9) circle (\vs);
\draw [fill=black] (13.6+\shiftD,-0.1) rectangle (13.8+\shiftD,0.1);
\draw [color=black] (13.35+\shiftD,0) node {\tiny $t+1$};
\draw [fill=black] (13.9+\shiftD,0.4) circle (\vs);
\draw [fill=black] (13.9+\shiftD,0.9) circle (\vs);
\draw [fill=black] (14.3+\shiftD,0.7) circle (\vs);
\draw [fill=black] (14.5+\shiftD,0.4) circle (\vs);
\draw [fill=black] (14.6+\shiftD,1) circle (\vs);
\draw [color=black] (14.38+\shiftD,1) node {\tiny $v$};
\draw [fill=black] (14.9+\shiftD,0.8) circle (\vs);
\draw [fill=black] (15+\shiftD,0.4) circle (\vs);
\draw [fill=black] (14.9+\shiftD,1.2) circle (\vs);
\draw [fill=black] (15.3+\shiftD,1) circle (\vs);
\draw [fill=black] (13.9+\shiftD,1.5) rectangle (14.1+\shiftD,1.7);
\draw [color=black] (13.85+\shiftD,1.38) node {\tiny $t+2$};
\draw [fill=black] (13.8+\shiftD,2.1) circle (\vs);
\draw [fill=black] (14.2+\shiftD,2.2) circle (\vs);
\draw [fill=black] (14.5+\shiftD,1.9) circle (\vs);
\draw [fill=black] (14.1+\shiftD,2.7) circle (\vs);
\draw [fill=black] (14.4+\shiftD,2.7) circle (\vs);
\draw [fill=black] (14.6+\shiftD,2.5) circle (\vs);
\draw [fill=black] (13.2+\shiftD,2.4) circle (\vs);
\draw [color=black] (12.98+\shiftD,2.4) node {\tiny $u$};
\draw [fill=black] (13.7+\shiftD,2.7) circle (\vs);
\draw [fill=black] (13.3+\shiftD,2.9) circle (\vs);
\draw [fill=black] (13.4+\shiftD,3.5) circle (\vs);
\draw [fill=black] (13.9+\shiftD,3.1) rectangle (14.1+\shiftD,3.3);
\draw [fill=black] (14+\shiftD,3.7) circle (\vs);
\draw [fill=black] (13.6+\shiftD,3.9) circle (\vs);
\draw [fill=black] (14.6+\shiftD,3.9) circle (\vs);
\draw [fill=black] (14.5+\shiftD,3.2) rectangle (14.7+\shiftD,3.4);

\draw [color=black] (0,3.965) node {\tiny $B_{e(t+1)}$};
\draw [color=black] (2.7+\shiftA,4.07) node {\tiny $B_{e(t+2)}$};
\draw [color=black] (7.05+\shiftB,3.04) node {\tiny $B_{e(t+1)}$};
\draw [color=black] (7.53+\shiftB,3.93) node {\tiny $B_{e(t+2)}$};
\draw [color=black] (11.7+\shiftC,4.15) node {\tiny $B_{e(t+1)}$};
\draw [color=black] (11.13+\shiftC,3.32) node {\tiny $B_{e(t+2)}$};
\end{scriptsize}
\end{tikzpicture}
        \caption{Construction of the (multi-)graph $F'$ by adding edges (dotted lines) between $u$ and $t+1$ and between $v$ and $t+2$ in the random forest $F=F(m,t+2)$. In the first three pictures we have $F'\in \mathcal{F}(m,t)$ and the number of vertices in the solid bubble is $\min\left\{B_{e(t+1)},B_{e(t+2)}\right\}$. The last picture is an example in which $F'\notin \mathcal{F}(m,t)$.}
    \label{fig:minBridgeNumber}
\end{figure}
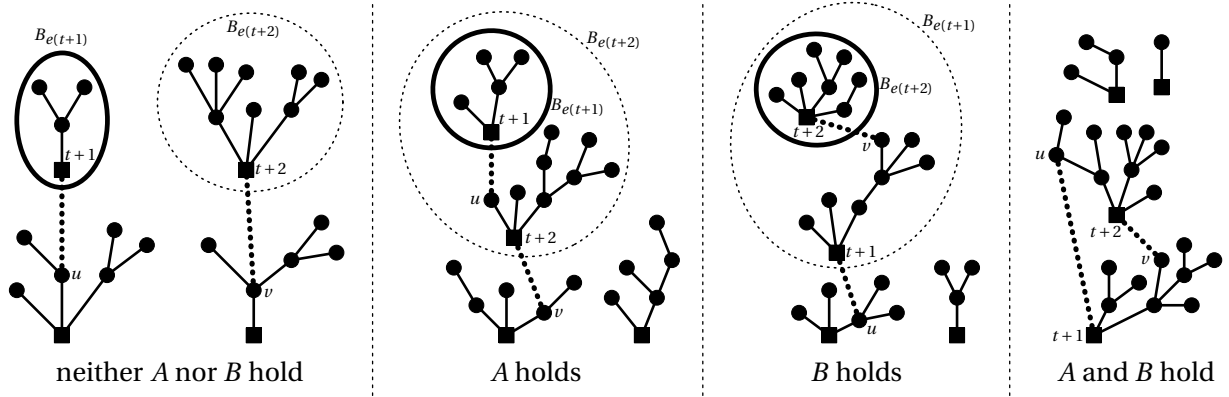

We note that for each forest $H\in \mathcal{F}(m,t)$ there is exactly one choice of $F$, $u$ and $v$ such that $F'=H$. Hence, conditioned on the event that $F'\in \mathcal{F}(m,t)$ we have that $F'$ is distributed as $F(m,t)$. Furthermore, we have
\begin{align}\label{eq:probCondition}
    \prob{F'\in \mathcal{F}(m,t)}=\frac{\left|\mathcal{F}(m,t)\right|}{\left|\mathcal{F}(m,t+2)\right|\cdot m\cdot m}
    \stackrel{\eqref{eq:numberofforests}}{=}
    \frac{t}{t+2}=1+o(1).
\end{align}
Next, we aim to relate the bridge numbers $B_{e(t+1)}$ and $B_{e(t+2)}$ to the above construction. To that end, we denote by $T_{t+1}$ and $T_{t+2}$ the tree components in $F$ containing $t+1$ and $t+2$, respectively. Furthermore, we denote by $A$ and $B$ the events that $u\in V\left(T_{t+2}\right)$ and $v\in V\left(T_{t+1}\right)$, respectively. We observe that if $F'\in \mathcal{F}(m,t)$, then $A$ and $B$ cannot hold simultaneously. Furthermore, conditioned on $F'\in \mathcal{F}(m,t)$ we have for the bridge numbers $B_{e(t+1)}$ and $B_{e(t+2)}$ in $F'$ and therefore also in $F(m,t)$ (see also \Cref{fig:minBridgeNumber})
\begin{align}\label{eq:bridgeNumber}
    M:=\min\left\{B_{e(t+1)},B_{e(t+2)}\right\}=
    \begin{cases}
    \min\left\{\order{T_{t+1}}, \order{T_{t+2}}\right\}& \text{~~if~neither } A \text{~nor~} B \text{~hold,}
    \\
    \order{T_{t+1}}& \text{~~if~} A \text{~holds,}
    \\
    \order{T_{t+2}}& \text{~~if~} B \text{~holds}.
    \end{cases}
\end{align}
We extend the definition of $M$ by setting $M:=0$ if $F'\notin \mathcal{F}(m,t)$. Using \eqref{eq:probCondition} we obtain
\begin{align}\label{eq:expectationBridgeNumbers}
   \expect{\min\left\{B_{e(t+1)},B_{e(t+2)}\right\}}&=\condexpect{M}{F'\in \mathcal{F}(m,t)} =\frac{\expect{M}}{\prob{F'\in \mathcal{F}(m,t)}}
   =(1+o(1))\expect{M}.
\end{align}
Furthermore, due to \eqref{eq:bridgeNumber} we obtain
\begin{align}\label{eq:expectationM}
    \expect{M}&\leq \expect{\min\big\{\order{T_{t+1}}, \order{T_{t+2}}\big\}}+\prob{A}\condexpect{\order{T_{t+1}}}{A}+\prob{B}\condexpect{\order{T_{t+2}}}{B}
\end{align}
Using \Cref{prop:expmincomponentsforest} we have
\begin{align}\label{eq:forest6A}
  \expect{\min\big\{\order{T_{t+1}}, \order{T_{t+2}}\big\}}=o(m/(t+2))=o(m/t). 
\end{align}
To bound $\condexpect{\order{T_{t+1}}}{A}$, let $H$ be a possible realisation of $T_{t+2}$ with $u\in V(H)$. Conditioned on the event $T_{t+2}=H$ the remaining forest $F\setminus T_{t+2}$ is a random forest on $m-\order{H}$ vertices having $t+1$ components. Thus, \Cref{rem:expectedtreesize}\eqref{r:expectedtreesize} implies
\begin{align*}
    \condexpect{\order{T_{t+1}}}{T_{t+2}=H}= \frac{m-\order{H}}{t+1}\leq \frac{m}{t}.
\end{align*}
As this is true for all possible realisations $H$ of $T_{t+2}$ that contain $u$, we obtain $\condexpect{\order{T_{t+1}}}{A}\leq m/t$. Combining this with \Cref{rem:expectedtreesize}\eqref{r:expectedtreesize} yields $\prob{A}\condexpect{\order{T_{t+1}}}{A}\leq \frac{1}{t+2}\cdot m/t=o(m/t)$. By symmetry, we also have $\prob{B}\condexpect{\order{T_{t+2}}}{B}=o(m/t)$. Together with \eqref{eq:expectationM} and \eqref{eq:forest6A} this shows that $\expect{M}=o(m/t)$. Finally, combining this with \eqref{eq:wlog} and \eqref{eq:expectationBridgeNumbers} yields the statement.
\end{proof}

\section{Proof of \Cref{thm:giantrainbowtree}}\label{s:weakly}

\subsection{The weakly subcritical regime}\label{s:weaklysub}

\begin{proof}[Proof of \Cref{thm:giantrainbowtree}(\ref{i:sub})]
By \Cref{t:giantLuczak}(\ref{i:Lsub}), whp the largest component in $G(n,p)$
    is a tree of order $(1+o(1))\frac{2}{\eps^2}\log\left(\eps^3 n\right)$, and in particular,
    this is therefore also trivially an upper bound on the order of the largest
    rainbow tree in $G_c(n,p)$. 
    
For the lower bound, conditioned on the likely event that the  largest component of $G(n,p)$ is a tree $T$ of order $(1+o(1))\frac{2}{\eps^2}\log\left(\eps^3 n\right)$, we simply delete all edges of $T$ whose colours appear more than once in $T$. We claim that whp this, slightly wasteful, process leaves a component covering almost all the vertices of $T$, and this component is then a rainbow tree in $G_c(n,p)$.

Indeed, conditioned on the largest component $T$ of $G(n,p)$ being a tree of order $k$, it is clear that (up to relabelling of vertices) $T$ is distributed as $F(k,1)$, a uniformly chosen random tree of order $k$. By \Cref{prop:unifedgedelete}, given a uniformly random chosen edge $e$ of $T$, the expected order of the smaller component in $T-e$ is $O\left(\sqrt{k}\right)$.

Furthermore, the expected number of pairs of edges in $T$ which receive the same colour is $O\left(k^2/c\right)$. Since we choose our colouring uniformly, it follows that the expected size of the largest component of $T$ after deleting all edges appearing in such pairs is at least
\[
|T| - O\left(\frac{k^2}{c}\right)\cdot O\left(\sqrt{k}\right)
= k -  O\left(k\cdot \frac{\left(\log\left(\eps^3n\right)\right)^{3/2}}{\eps^3 n}\right)=(1-o(1))k.
\]
Hence, by Markov's inequality (applied to the \emph{complement} of the largest component), whp $T$ will contain a rainbow tree of order $(1-o(1))k = (1+o(1))\frac{2}{\eps^2}\log\left(\eps^3 n\right)$, as required.
\end{proof}

We note that the proof above actually shows that if we uniformly colour a tree with $k$ vertices with $c = \omega\left(k^{\frac{3}{2}}\right)$ colours, then whp there will be a rainbow subtree of order $(1-o(1))k$. A first consequence of this is that \Cref{thm:giantrainbowtree}(\ref{i:sub}) holds even with significantly fewer colours -- here we can take any $c =\omega\left(\frac{\left(\log(\eps^3 n)\right)^{3/2}}{\eps^3}\right)$. A second consequence is that the above proof also holds without modification in the strictly subcritical regime, where $\eps \in (0,1)$ is a constant, for any $c= \omega\left(( \log n)^{\frac{3}{2}}\right)$. Indeed, by \Cref{t:giantsparse} whp the largest component in this regime is a tree of order $k=O(\log n)$ and so the estimates above will still suffice to show the existence of a rainbow subtree of order $(1-o(1))k$.

\subsection{The weakly supercritical regime}\label{s:weaklysup}

Our aim in this section is to prove \Cref{thm:giantrainbowtree}(\ref{i:sup}),
for which we first collect some preliminary results.
Throughout this section let us fix $\alpha > 0$, $c=\alpha n$ and $\eps= \eps(n)$ such that $\eps^3n \to \infty$ and $\eps=o(1)$, and set $p=\frac{1+\eps}{n}$. We will think of the coloured graph $G_c(n,p)$ as a pair $(G(n,p),\chi)$ where $G(n,p)$ is a random graph and $\chi$ is a uniformly chosen random colouring of the edges of $G(n,p)$ with colours from $[c]$.

We will first be interested in properties of  the \emph{$2$-core} $C$ of $G(n,p)$, the unique maximal subgraph of minimum degree at least two. Let $\bm{d}$ be the degree sequence  of $C$. The properties of $C$ in the weakly supercritical regime are quite well understood. In particular, work of {\L}uczak \cite{L90,L91} implies the following:
\begin{lemma}[\cite{L90,L91}]\label{lem:coreproperties}
Let $C$ be the $2$-core of $G(n,p)$.
\begin{enumerate}[(a)]
\item\label{i: C} Whp $C$ is of order and size $\Theta\left(\eps^2 n\right)$.
\item\label{i: Cdeg3} Whp $C$ contains $\Theta\left(\eps^3 n\right)$ vertices of degree three and has $o\left(\eps^3 n\right)$ edges incident to vertices of degree at least four. 
\item \label{i:unifcore} Conditional on its degree sequence $\bm{d}$, $C$ is distributed uniformly at random from all simple graphs on $V(C)$ with degree sequence $\bm{d}$.
\item\label{i:contiguity} Whp the degree sequence $\bm{d}$ is such that the multigraph $G^*(\bm{d})$ drawn from the configuration model is simple with probability
$\Theta(1)$.
\end{enumerate}
\end{lemma}

In what follows, we will assume that we have conditioned on some specific values of $V(C)$ and $\bm{d}$ which satisfy the conclusions of \Cref{lem:coreproperties}; for simplicity, we will suppress this conditioning in the notation.

We generate a list $L' \in [c]^{e(C)}$ of colours of length $e(C)$ uniformly at random and let $X$ be the number of elements of $L'$ which are not unique,
i.e.\ the number of coordinates in which a colour appears in $L'$ that also appears in some other cooordinate.

\begin{claim}\label{claim:colourclash}
Whp $X = o\left(\eps^3 n\right)$.
\end{claim}
\begin{proof}
Note that it follows from \Cref{lem:coreproperties} and our conditioning that $e(C) = \Theta\left(\eps^2 n\right)$. Then, it is easy to see that the probability that a single element is not unique is 
\[
1-\left(1-\frac{1}{c}\right)^{e(C)-1} =1 - 1+ \frac{e(C)-1}{\alpha n} + O\left( \left(\frac{e(C)}{\alpha n}\right)^2 \right)=  \Theta\left(\eps^2\right).
\]
Hence $\mathbb{E}(X) = \Theta\left(\eps^2\right) |L'| = \Theta\left(\eps^4n\right)$ and the conclusion holds by Markov's inequality.
\end{proof}

We now condition on the value of $X$, which we may assume satisfies the conclusion of \Cref{claim:colourclash}. We let $L$ be the re-ordering of $L'$ in which the repeated elements appear at the start and we generate a pair
$\left(G^*(\bm{d}),\hat{\chi}\right)$ whose joint distribution is given as follows:
\begin{itemize}
    \item $G^*(\bm{d})$ is generated according to the configuration model with degree sequence $\bm{d}$. Let us assume that we expose the edges in $G^*(\bm{d})$ in some random order $e_1, \ldots, e_{e(C)}$.
    \item $\hat{\chi}$ is given by assigning the colour $L_i$ to $e_i$ for each $i$.
\end{itemize}
Note that, by \Cref{lem:coreproperties}\eqref{i:unifcore}
\begin{equation}\label{e:dchi}
\text{Conditioned on $G^*(\bm{d})$ being simple, $\left(G^*(\bm{d}), \hat{\chi}\right)$
has the same distribution as $\left(C,\chi|_{E(C)}\right)$.}
\end{equation}

The edges $R' := \{e_1,\ldots, e_X\}$ are those whose colours in $\hat{\chi}$ are not unique, and we are interested in the graph remaining when we delete $R'$. Let $\bm{\hat{d}}$ be the degree sequence of this graph, that is, for each vertex $i \in V(C)$ we have 
\[
\hat{d}_i = d_i - |\{j \colon j \leq X \text{ and } i \in e_j\}|,
\]
where loops are counted with multiplicity two.

By the principle of deferred decisions, it is clear that $G^*(\bm{d}) - R' \sim G^*(\bm{\hat{d}})$ and furthermore if we write $\hat{D}_i := \left|\left\{ j \colon \hat{d}_j = i\right\}\right|$, then by \Cref{lem:coreproperties}\eqref{i: C} and \eqref{i: Cdeg3}, and \Cref{claim:colourclash}, it follows that
\begin{equation} \label{eq:coredegs}
\hat{D}_0,\hat{D}_1 = o\left(\eps^3 n\right), \qquad \hat{D}_2 = \Theta\left(\eps^2n\right) \qquad \text{ and } \qquad \sum_{i\geq 3} \hat{D}_i = \Theta\left(\eps^3 n\right).
\end{equation}

Let $R$ be the set of edges of $C$ which do not receive unique colours in $\chi$.

\begin{claim}\label{claim:C-R}
Any graph property which holds whp for $G^*(\bm{\hat{d}})$ also holds whp for $C-R$.
\end{claim}
\begin{proof}
We generate $G^*(\bm{d}) \supseteq G^*(\bm{\hat{d}})$ and
recall that \Cref{lem:coreproperties}\eqref{i:contiguity}
states that the probability that $G^*(\bm{d})$ is simple is $\Theta(1)$. 
It follows from \eqref{e:dchi} that conditioned on this event $C-R \sim G^*(\bm{\hat{d}})$ and so the claim follows.  
\end{proof}

We then use the following result,
which is a special case of a result of {\L}uczak~\cite[Theorem~2]{Luczak}.

\begin{lemma}\label{lem:Luczak}
Let $\bm{\hat{d}}=(\hat{d}_1,\ldots,\hat{d}_{|V(C)|}) \in \mathbb{N}^{|V(C)|}$ be a degree sequence in which \eqref{eq:coredegs} holds.
Then whp $G^*(\bm{\hat{d}})$ contains a component of order $(1-o(1))|V(C)|$.
\end{lemma}

By \Cref{claim:C-R} and \Cref{lem:Luczak}, it follows that whp $C-R$ contains a component $\hat{C}$ of order $(1-o(1))|V(C)| = \Theta\left(\eps^2 n\right)$. Note that, by construction, this component is rainbow coloured; however it is still significantly smaller than $2\eps n$. In order to extend this to a larger rainbow subgraph, we will consider how the rest of the giant component is distributed around its $2$-core. The component structure of $G(n,p)$ in the weakly supercritical regime is reasonably well-understood, and in particular work of {\L}uczak \cite{L90} implies the following: 

\begin{lemma}[\cite{L90}]\label{lem:compstructure}
Let $L$ be the vertex set of the largest component of $G(n,p)$, let $C$ be the $2$-core of $G(n,p)$ and let $U$ be the set of vertices contained in unicyclic components. 
\begin{enumerate}[(a)]
    \item\label{i:sizeLU} Whp $|L| = (2+o(1))\eps n$ and $|U| = \Theta\left(\frac{1}{\eps^2}\right) = o(\eps n)$;
    \item\label{i: unique} Whp the largest component is the unique component with more than one cycle, and hence $V(C) \subseteq L \cup U$.
    \item Conditioned on \eqref{i: unique}, $G(n,p)[L\cup U] - E(C) \sim F(L \cup U, V(C))$,
\end{enumerate}
where $F(X,Y)$ is a uniformly chosen random forest whose vertex set is $X$, which contains $|Y|$ trees and in which each element of $Y$ lies in a different tree.
\end{lemma}

In what follows, let us condition on some specific values of $L$ and $U$ and the fact that the high probability events of \Cref{lem:compstructure}\eqref{i:sizeLU} and~\eqref{i: unique} hold. In particular, if we let $F = G(n,p)[L\cup U] - C$ then\footnote{Recall that $F(m,t)$ is a random forest with $m$ vertices and $t$ components. Formally here we are relabelling the vertices of $G(n,p)$ so that $L \cup U = [m]$ and $V(C) =[t]$.} $F \sim F(m,t)$ where $m = (2+o(1))\eps n$ and $t = \Theta\left(\eps^2 n\right)$.
We can now complete our analysis of the weakly supercritical case.

\begin{proof}[Proof of \Cref{thm:giantrainbowtree}(\ref{i:sup})]

Let us now consider the colouring $\chi|_F$. We know that $\chi|_{\hat{C}}$ is rainbow and uses a set $Z$ of $\Theta\left(\eps^2 n\right)$ colours. Recall that $B_e$, the \emph{bridge-number} of $e$, is the number of vertices in $F$ disconnected from the root of their component when we delete $e$, i.e. the number of vertices contained in the branch of $e$. We build a rainbow subgraph of $F$ which shares no colour with $Z$ as follows:
\begin{enumerate}[(1)]
\item\label{i:colshared} For each colour $i \in Z$ and each edge $e$ in $F$ of colour $i$, we delete the branch of $e$;
\item\label{i:colthree} For each colour $i \in [c]$ which appears at least three times in $F$ and each edge $e$ of colour $i$, we delete the branch of $e$;
\item\label{i:coltwo} For each colour $i \in [c]$ which appears on exactly two edges $e$ and $f$ of $F$, where wlog $B_e \leq B_f$, we delete the branch of $e$;
\item\label{i:notinhatC} We delete all vertices in components whose roots are not in $V(\hat{C})$.
\end{enumerate}
See \Cref{fig: cut forest} for an example of this construction.

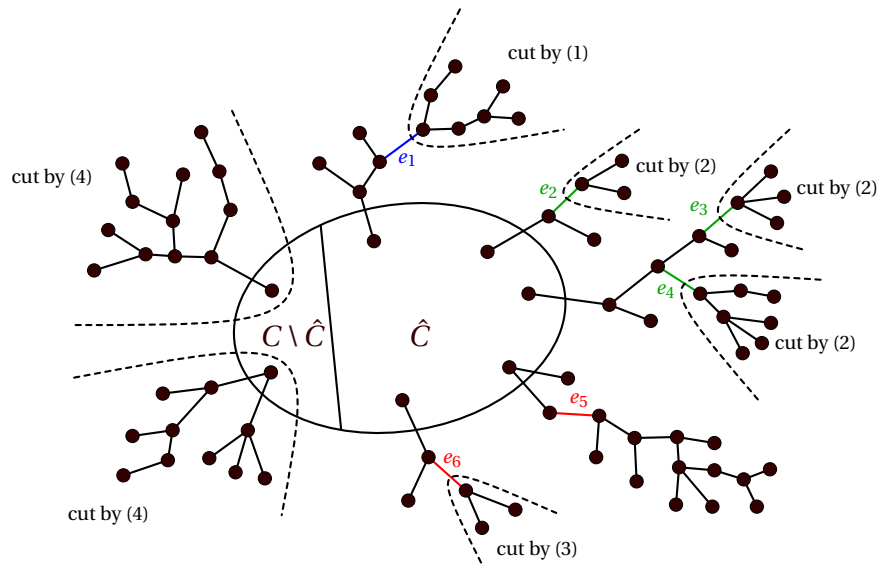
\begin{figure}[h]
    \centering
   \definecolor{ududff}{rgb}{0.2,0.,0.}
\begin{tikzpicture}[line cap=round,line join=round,>=triangle 45,x=1.0cm,y=1.0cm]
\clip(1.,0.) rectangle (14.,8.);
\draw [rotate around={9.6705181649247:(7.018630879943382,3.493293893284993)},line width=0.8pt] (7.018630879943382,3.493293893284993) ellipse (2.220480932695539cm and 1.5046124421558438cm);
\draw [line width=0.8pt] (5.9682381521604455,4.730625106801324)-- (6.24096087631946,2.0039070251051267);
\draw [line width=0.8pt] (8.731500068580912,3.8176360742560385)-- (9.805925512791445,3.6699025756770913);
\draw [line width=0.8pt] (9.805925512791445,3.6699025756770913)-- (10.450580779317765,4.180254661677091);
\draw [line width=0.8pt] (9.805925512791445,3.6699025756770913)-- (10.356568552949343,3.4550174868349863);
\draw [line width=0.8pt] (10.450580779317765,4.180254661677091)-- (11.001223819475662,4.583164203256038);
\draw [color=darkgreen,line width=0.8pt] (10.450580779317765,4.180254661677091)-- (11.014654137528293,3.8176360742560385);
\draw [line width=0.8pt] (5.320199283212473,3.857927028413933)-- (4.514380200054573,4.3011275241507745);
\draw [line width=0.8pt] (4.514380200054573,4.3011275241507745)-- (4.030888750159834,4.314557842203406);
\draw [line width=0.8pt] (4.514380200054573,4.3011275241507745)-- (4.769556243054574,4.9323524726244585);
\draw [line width=0.8pt] (4.030888750159834,4.314557842203406)-- (4.00402811405457,4.784618974045511);
\draw [line width=0.8pt] (4.030888750159834,4.314557842203406)-- (3.6414095266335154,4.341418478308669);
\draw [line width=0.8pt] (4.00402811405457,4.784618974045511)-- (4.138331294580887,5.402413604466563);
\draw [line width=0.8pt] (4.00402811405457,4.784618974045511)-- (3.466815391949304,5.039795017045511);
\draw [line width=0.8pt] (4.769556243054574,4.9323524726244585)-- (4.621822744475626,5.442704558624458);
\draw [line width=0.8pt] (4.621822744475626,5.442704558624458)-- (4.380077019528256,5.966486962677088);
\draw [line width=0.8pt] (3.6414095266335154,4.341418478308669)-- (3.1444877586861444,4.663746111571827);
\draw [line width=0.8pt] (3.6414095266335154,4.341418478308669)-- (2.9564633059493013,4.126533389466564);
\draw [line width=0.8pt] (3.466815391949304,5.039795017045511)-- (3.3325122114229875,5.55014710304551);
\draw [line width=0.8pt] (5.306768965159841,2.7700712661507767)-- (4.514380200054573,2.5686164953613035);
\draw [line width=0.8pt] (5.306768965159841,2.7700712661507767)-- (5.,2.);
\draw [line width=0.8pt] (4.514380200054573,2.5686164953613035)-- (3.8697249335282535,2.4880345870455143);
\draw [line width=0.8pt] (4.514380200054573,2.5686164953613035)-- (4.,2.);
\draw [line width=0.8pt] (4.,2.)-- (3.466815391949304,1.9239612288349885);
\draw [line width=0.8pt] (4.,2.)-- (3.936876523791412,1.601633595571831);
\draw [line width=0.8pt] (3.936876523791412,1.601633595571831)-- (3.3459425294756193,1.4001788247823577);
\draw [line width=0.8pt] (5.,2.)-- (4.48751956394931,1.6284942316770943);
\draw [line width=0.8pt] (5.,2.)-- (4.836707833317733,1.4404697789402523);
\draw [line width=0.8pt] (5.,2.)-- (5.226187056844051,1.359887870624463);
\draw [line width=0.8pt] (6.676661406528269,4.5160126129928795)-- (6.487293921986163,5.171412133961296);
\draw [line width=0.8pt] (6.487293921986163,5.171412133961296)-- (5.95142423168616,5.55014710304551);
\draw [line width=0.8pt] (6.487293921986163,5.171412133961296)-- (6.75321421942827,5.570292580124454);
\draw [line width=0.8pt] (6.75321421942827,5.570292580124454)-- (6.487293921986163,5.954399676429716);
\draw [color=blue,line width=0.8pt] (6.75321421942827,5.570292580124454)-- (7.329374863886168,5.9987197260034);
\draw [line width=0.8pt] (7.329374863886168,5.9987197260034)-- (7.432788312891432,6.456693571598136);
\draw [line width=0.8pt] (7.329374863886168,5.9987197260034)-- (7.802122059338802,6.013493075861295);
\draw [line width=0.8pt] (7.802122059338802,6.013493075861295)-- (8.141909106070383,6.175999924298137);
\draw [line width=0.8pt] (7.432788312891432,6.456693571598136)-- (7.7578020097651175,6.840800667903398);
\draw [line width=0.8pt] (8.141909106070383,6.175999924298137)-- (8.393056053654595,6.574880370461294);
\draw [line width=0.8pt] (8.141909106070383,6.175999924298137)-- (8.599882951665123,6.1464532245823476);
\draw [line width=0.8pt] (8.186229155644067,4.373651241634983)-- (8.998763397828283,4.846398437087613);
\draw [line width=0.8pt] (8.998763397828283,4.846398437087613)-- (9.60447074200197,4.536158090071824);
\draw [color=darkgreen,line width=0.8pt] (11.001223819475662,4.583164203256038)-- (11.511575905475665,5.026364698992879);
\draw [line width=0.8pt] (11.001223819475662,4.583164203256038)-- (11.430993997159876,4.395139750519196);
\draw [line width=0.8pt] (11.014654137528293,3.8176360742560385)-- (11.55186685963356,3.857927028413933);
\draw [line width=0.8pt] (11.014654137528293,3.8176360742560385)-- (11.323551452738823,3.508738759045513);
\draw [color=darkgreen, line width=0.8pt] (8.998763397828283,4.846398437087613)-- (9.441963893565127,5.274825582966559);
\draw [line width=0.8pt] (11.323551452738823,3.508738759045513)-- (11.578727495738823,3.0252473091507763);
\draw [line width=0.8pt] (11.323551452738823,3.508738759045513)-- (11.833903538738825,3.159550489677092);
\draw [line width=0.8pt] (11.323551452738823,3.508738759045513)-- (11.95477640121251,3.4281568507297235);
\draw [line width=0.8pt] (11.55186685963356,3.857927028413933)-- (11.995067355370404,3.777345120098144);
\draw [line width=0.8pt] (8.476324025580912,2.8372228564139346)-- (9.255282472633548,2.6894893578349874);
\draw [line width=0.8pt] (8.476324025580912,2.8372228564139346)-- (9.013536747686178,2.2328585440455146);
\draw [color=red, line width=0.8pt] (9.013536747686178,2.2328585440455146)-- (9.671622332265128,2.1925675898876196);
\draw [line width=0.8pt] (9.671622332265128,2.1925675898876196)-- (9.631331378107234,1.6419245497297257);
\draw [line width=0.8pt] (9.671622332265128,2.1925675898876196)-- (10.141683464107237,1.8971005927297253);
\draw [line width=0.8pt] (10.141683464107237,1.8971005927297253)-- (10.705756822317765,1.910530910782357);
\draw [line width=0.8pt] (10.141683464107237,1.8971005927297253)-- (10.1685441002125,1.3195969164665684);
\draw [line width=0.8pt] (10.705756822317765,1.910530910782357)-- (10.732617458423029,1.5076213692034102);
\draw [line width=0.8pt] (10.705756822317765,1.910530910782357)-- (11.202678590265137,1.8299490024665677);
\draw [line width=0.8pt] (10.732617458423029,1.5076213692034102)-- (11.202678590265137,1.4673304150455155);
\draw [line width=0.8pt] (10.732617458423029,1.5076213692034102)-- (10.692326504265134,1.0106996012560423);
\draw [line width=0.8pt] (10.732617458423029,1.5076213692034102)-- (11.175817954159873,0.9838389651507793);
\draw [line width=0.8pt] (11.202678590265137,1.4673304150455155)-- (11.592157813791454,1.3464575525718314);
\draw [line width=0.8pt] (11.592157813791454,1.3464575525718314)-- (11.766751948475667,0.9838389651507793);
\draw [line width=0.8pt] (11.592157813791454,1.3464575525718314)-- (11.941346083159878,1.480760733098147);
\draw [line width=0.8pt] (7.054165082791432,2.4000101343086726)-- (7.401898581370379,1.6419245497297257);
\draw [line width=0.8pt] (7.401898581370379,1.6419245497297257)-- (7.133292220317745,1.0644208734665686);
\draw [color=red, line width=0.8pt] (7.401898581370379,1.6419245497297257)-- (7.89882034931775,1.1987240539928843);
\draw [line width=0.8pt] (7.89882034931775,1.1987240539928843)-- (8.113705438159856,0.7018022860455165);
\draw [line width=0.8pt] (7.89882034931775,1.1987240539928843)-- (8.5569059338967,0.9435480109928847);
\draw [line width=0.8pt] (11.511575905475665,5.026364698992879)-- (12.12937053589672,5.106946607308668);
\draw [line width=0.8pt] (11.511575905475665,5.026364698992879)-- (11.95477640121251,5.442704558624458);
\draw [line width=0.8pt] (9.441963893565127,5.274825582966559)-- (9.97380448844934,5.585065929982348);
\draw [line width=0.8pt] (9.441963893565127,5.274825582966559)-- (10.00335118816513,5.156638784103402);
\draw [line width=0.8pt] (11.511575905475665,5.026364698992879)-- (12.0353583095283,4.757758337940248);
\draw [line width=0.8pt, dash pattern= on 2 pt off 2.pt]  (7.82,7.6) .. controls (6.75,5.57) .. (9.2,6);
\draw [line width =0.8pt, dash pattern= on 2 pt off 2.pt] (10.2,6).. controls (8.8,5.05).. (10.6,4.8);
\draw [line width =0.8pt, dash pattern= on 2 pt off 2.pt] (12.2,6).. controls (10.9,4.8).. (12.4,4.4);
\draw [line width =0.8pt, dash pattern= on 2 pt off 2.pt] (11.78,2.41)  .. controls (10.3,4.15)..(12.66,4) ;
\draw [line width =0.8pt, dash pattern= on 2 pt off 2.pt] (8.1,.24)  .. controls (7.4,1.64)..(9.1,0.9) ;
\draw [line width =0.8pt, dash pattern= on 2 pt off 2.pt] (4.79,6.25)  .. controls (6.1,3.4)..(2.71,3.4) ;
\draw [line width =0.8pt, dash pattern= on 2 pt off 2.pt] (2.69,2.7)  .. controls (5.9,3.3)..(5.45,0.87) ;
\begin{scriptsize}
\draw [fill=ududff] (8.731500068580912,3.8176360742560385) circle (2.5pt);
\draw [fill=ududff] (9.805925512791445,3.6699025756770913) circle (2.5pt);
\draw [fill=ududff] (10.450580779317765,4.180254661677091) circle (2.5pt);
\draw [fill=ududff] (10.356568552949343,3.4550174868349863) circle (2.5pt);
\draw [fill=ududff] (11.001223819475662,4.583164203256038) circle (2.5pt);
\draw [fill=ududff] (11.014654137528293,3.8176360742560385) circle (2.5pt);
\draw[color=darkgreen] (10.56,3.85) node {$e_4$};
\draw[color=black] (2.4,5.36) node {$\textrm{cut by $(4)$}$};
\draw[color=black] (3.15,0.87) node {$\textrm{cut by $(4)$}$};
\draw[color=black] (8.85,0.4) node {$\textrm{cut by $(3)$}$};
\draw[color=black] (9,7) node {$\textrm{cut by $(1)$}$};
\draw[color=black] (10.7,5.5) node {$\textrm{cut by $(2)$}$};
\draw[color=black] (12.84,5.21) node {$\textrm{cut by $(2)$}$};
\draw[color=black] (12.55,3.15) node {$\textrm{cut by $(2)$}$};
\draw [fill=ududff] (5.320199283212473,3.857927028413933) circle (2.5pt);
\draw [fill=ududff] (4.514380200054573,4.3011275241507745) circle (2.5pt);
\draw [fill=ududff] (4.030888750159834,4.314557842203406) circle (2.5pt);
\draw [fill=ududff] (4.769556243054574,4.9323524726244585) circle (2.5pt);
\draw [fill=ududff] (4.00402811405457,4.784618974045511) circle (2.5pt);
\draw [fill=ududff] (3.6414095266335154,4.341418478308669) circle (2.5pt);
\draw [fill=ududff] (4.138331294580887,5.402413604466563) circle (2.5pt);
\draw [fill=ududff] (3.466815391949304,5.039795017045511) circle (2.5pt);
\draw [fill=ududff] (4.621822744475626,5.442704558624458) circle (2.5pt);
\draw [fill=ududff] (4.380077019528256,5.966486962677088) circle (2.5pt);
\draw [fill=ududff] (3.1444877586861444,4.663746111571827) circle (2.5pt);
\draw [fill=ududff] (2.9564633059493013,4.126533389466564) circle (2.5pt);
\draw [fill=ududff] (3.3325122114229875,5.55014710304551) circle (2.5pt);
\draw [fill=ududff] (5.306768965159841,2.7700712661507767) circle (2.5pt);
\draw [fill=ududff] (4.514380200054573,2.5686164953613035) circle (2.5pt);
\draw [fill=ududff] (5.,2.) circle (2.5pt);
\draw [fill=ududff] (3.8697249335282535,2.4880345870455143) circle (2.5pt);
\draw [fill=ududff] (4.,2.) circle (2.5pt);
\draw [fill=ududff] (3.466815391949304,1.9239612288349885) circle (2.5pt);
\draw [fill=ududff] (3.936876523791412,1.601633595571831) circle (2.5pt);
\draw [fill=ududff] (3.3459425294756193,1.4001788247823577) circle (2.5pt);
\draw [fill=ududff] (4.48751956394931,1.6284942316770943) circle (2.5pt);
\draw [fill=ududff] (4.836707833317733,1.4404697789402523) circle (2.5pt);
\draw [fill=ududff] (5.226187056844051,1.359887870624463) circle (2.5pt);
\draw [fill=ududff] (6.676661406528269,4.5160126129928795) circle (2.5pt);
\draw [fill=ududff] (6.487293921986163,5.171412133961296) circle (2.5pt);
\draw [fill=ududff] (5.95142423168616,5.55014710304551) circle (2.5pt);
\draw [fill=ududff] (6.75321421942827,5.570292580124454) circle (2.5pt);
\draw [fill=ududff] (6.487293921986163,5.954399676429716) circle (2.5pt);
\draw [fill=ududff] (7.329374863886168,5.9987197260034) circle (2.5pt);
\draw[color=blue] (7.122547965875641,5.6) node {$e_1$};
\draw [fill=ududff] (7.432788312891432,6.456693571598136) circle (2.5pt);
\draw [fill=ududff] (7.802122059338802,6.013493075861295) circle (2.5pt);
\draw [fill=ududff] (8.141909106070383,6.175999924298137) circle (2.5pt);
\draw [fill=ududff] (7.7578020097651175,6.840800667903398) circle (2.5pt);
\draw [fill=ududff] (8.393056053654595,6.574880370461294) circle (2.5pt);
\draw [fill=ududff] (8.599882951665123,6.1464532245823476) circle (2.5pt);
\draw [fill=ududff] (8.186229155644067,4.373651241634983) circle (2.5pt);
\draw [fill=ududff] (8.998763397828283,4.846398437087613) circle (2.5pt);
\draw [fill=ududff] (9.60447074200197,4.536158090071824) circle (2.5pt);
\draw [fill=ududff] (11.511575905475665,5.026364698992879) circle (2.5pt);
\draw[color=darkgreen] (11,4.95) node {$e_3$};
\draw [fill=ududff] (11.430993997159876,4.395139750519196) circle (2.5pt);
\draw [fill=ududff] (11.55186685963356,3.857927028413933) circle (2.5pt);
\draw [fill=ududff] (11.323551452738823,3.508738759045513) circle (2.5pt);
\draw [fill=ududff] (9.441963893565127,5.274825582966559) circle (2.5pt);
\draw[color=darkgreen] (9,5.18) node {$e_2$};
\draw [fill=ududff] (11.578727495738823,3.0252473091507763) circle (2.5pt);
\draw [fill=ududff] (11.833903538738825,3.159550489677092) circle (2.5pt);
\draw [fill=ududff] (11.95477640121251,3.4281568507297235) circle (2.5pt);
\draw [fill=ududff] (11.995067355370404,3.777345120098144) circle (2.5pt);
\draw [fill=ududff] (8.476324025580912,2.8372228564139346) circle (2.5pt);
\draw [fill=ududff] (9.255282472633548,2.6894893578349874) circle (2.5pt);
\draw [fill=ududff] (9.013536747686178,2.2328585440455146) circle (2.5pt);
\draw [fill=ududff] (9.671622332265128,2.1925675898876196) circle (2.5pt);
\draw[color=red] (9.405030518920391,2.4) node {$e_5$};
\draw [fill=ududff] (9.631331378107234,1.6419245497297257) circle (2.5pt);
\draw [fill=ududff] (10.141683464107237,1.8971005927297253) circle (2.5pt);
\draw [fill=ududff] (10.705756822317765,1.910530910782357) circle (2.5pt);
\draw [fill=ududff] (10.1685441002125,1.3195969164665684) circle (2.5pt);
\draw [fill=ududff] (10.732617458423029,1.5076213692034102) circle (2.5pt);
\draw [fill=ududff] (11.202678590265137,1.8299490024665677) circle (2.5pt);
\draw [fill=ududff] (11.202678590265137,1.4673304150455155) circle (2.5pt);
\draw [fill=ududff] (10.692326504265134,1.0106996012560423) circle (2.5pt);
\draw [fill=ududff] (11.175817954159873,0.9838389651507793) circle (2.5pt);
\draw [fill=ududff] (11.592157813791454,1.3464575525718314) circle (2.5pt);
\draw [fill=ududff] (11.766751948475667,0.9838389651507793) circle (2.5pt);
\draw [fill=ududff] (11.941346083159878,1.480760733098147) circle (2.5pt);
\draw [fill=ududff] (7.054165082791432,2.4000101343086726) circle (2.5pt);
\draw [fill=ududff] (7.401898581370379,1.6419245497297257) circle (2.5pt);
\draw [fill=ududff] (7.133292220317745,1.0644208734665686) circle (2.5pt);
\draw [fill=ududff] (7.89882034931775,1.1987240539928843) circle (2.5pt);
\draw[color=red] (7.72,1.6) node {$e_6$};
\draw [fill=ududff] (8.113705438159856,0.7018022860455165) circle (2.5pt);
\draw [fill=ududff] (8.5569059338967,0.9435480109928847) circle (2.5pt);
\draw[color=ududff] (5.6,3.3) node {\large$C\setminus \hat{C}$};
\draw[color=ududff] (7.29,3.3) node {\large$\hat{C}$};
\draw [fill=ududff] (12.12937053589672,5.106946607308668) circle (2.5pt);
\draw [fill=ududff] (11.95477640121251,5.442704558624458) circle (2.5pt);
\draw [fill=ududff] (9.97380448844934,5.585065929982348) circle (2.5pt);
\draw [fill=ududff] (10.00335118816513,5.156638784103402) circle (2.5pt);
\draw [fill=ududff] (12.0353583095283,4.757758337940248) circle (2.5pt);
\end{scriptsize}
\end{tikzpicture}
    \caption{In this example, the edge $\color{blue}{e_1}$ receives a colour (blue) in $Z$ and so shares a colour with an edge in $\hat{C}$; the edges $\color{darkgreen} e_2$, $\color{darkgreen} e_3$ and $\color{darkgreen}e_4$ all receive the same colour (green); and the edges $\color{red}e_5$ and $\color{red} e_6$ receive the same colour (red), which does not appear elsewhere in $F$, and $B_{\color{red}e_5}\geq B_{\color{red}e_6}$.}
    \label{fig: cut forest}
\end{figure}

Let $\hat{F}$ be the remaining forest. Clearly by construction, $\hat{C} \cup \hat{F}$ is connected and rainbow coloured.

Let us write $X_*$ for the number of vertices we delete in step  $(*)$ for $* \in \{1,2,3,4\}$, respectively. Our aim will be to show that $\sum_{i=1}^4 X_i$ is negligible, and so the majority of vertices in $C$ lie in $\hat{C} \cup \hat{F}$.

We begin with some observations which will be required for the cases when
$*=$ \ref{i:colshared} or \ref{i:colthree}, for which we define $E_*$ to be the set of edges whose branch is deleted in step~$(*)$. Then clearly
\[
X_* \leq \sum_{e \in E_*} B_e.
\]
Furthermore, since the colouring $\chi$ is chosen independently of the forest $F$, we can see that
$\mathbb{E}(X_*) \leq \mathbb{E}_{\chi}( |E_*|) \cdot \mathbb{E}_{F,e}(B_e)$,
where $e$ is an edge of $F$ chosen uniformly at random.
By \Cref{claim:expbridgenumberforest}, $\mathbb{E}_{F,e}(B_e) \leq \frac{m}{t+1} = O\left(\frac{1}{\eps}\right)$,
so we have
\begin{equation}\label{eq:expXstar}
    \mathbb{E}(X_*)= O\left(\frac{1}{\eps}\right)\cdot \mathbb{E}_\chi(|E_*|).
\end{equation}

\subsection*{Colours appearing in $\hat C$ and $F$: Case $*=$ \ref{i:colshared}.}

We have
\begin{align*}
\mathbb{E}_{\chi}\left( \left|E_{\ref{i:colshared}}\right|\right) &= \sum_{i \in Z} \sum_{k \geq 1} k\cdot \mathbb{P}(\text{colour } i \text{ appears } k \text{ times in } F)\\
&\leq |Z| \sum_{k \geq 1} k \binom{m-t}{k} \frac{1}{c^k} \left(1 - \frac{1}{c}\right)^{m-t - k}\\
&\leq \Theta\left(\eps^2 n\right) \sum_{k \geq 1}  k\left(\frac{m}{c}\right)^k \\
&\leq \Theta\left(\eps^2 n\right) \sum_{k \geq 1}  k\left(\frac{3 \eps}{\alpha}\right)^k 
= O\left(\eps^3 n\right),
\end{align*}
and so $\mathbb{E}(X_{\ref{i:colshared}}) = O\left(\eps^2 n\right)$
by~\eqref{eq:expXstar}.

\subsection*{High-frequency colours: Case $*=$ \ref{i:colthree}.}

Similarly we have
\begin{align*}
\mathbb{E}_{\chi}\left( \left|E_{\ref{i:colthree}}\right|\right) &= \sum_{i \in [c]} \sum_{k \geq 3} k \cdot \mathbb{P}(\text{colour } i \text{ appears } k \text{ times in } F)\\
&=c \sum_{k \geq 3} k \binom{m-t}{k} \frac{1}{c^k} \left(1 - \frac{1}{c}\right)^{m-t - k}\\
&\leq c \sum_{k \geq 3} k \left( \frac{m}{c}\right)^k
= O\left(\eps^3 n\right),
\end{align*}
and so $\mathbb{E}\left(X_{\ref{i:colthree}}\right) = O\left(\eps^2 n\right)$ by~\eqref{eq:expXstar}.

\subsection*{Double colours: Case $*=$ \ref{i:coltwo}.}

Next, let $A_{\ref{i:coltwo}}$
denote the set of \emph{pairs} of distinct edges $e,f \in E(F)$ which share a colour which only appears on these two edges.
Then we have that
\begin{align*}
\mathbb{E}\left(X_{\ref{i:coltwo}}\right) \leq \mathbb{E}_{\chi}\left( \left|A_{\ref{i:coltwo}}\right|\right) \cdot \mathbb{E}_{F,e,f}\left(\min\left\{ B_e,B_f\right\}\right).
\end{align*}
Furthermore, by \Cref{prop:expminbnforest}
\[
\mathbb{E}_{F,e,f}\left(\min\left\{ B_e,B_f\right\}\right) = o\left(\frac{m}{t}\right) = o\left( \frac{1}{\eps}\right)
\]
and 
\begin{align*}
    \mathbb{E}_{\chi}\left( \left|A_{\ref{i:coltwo}}\right|\right) &= c \binom{m-t}{2}\frac{1}{c^2} \left( 1- \frac{1}{c}\right)^{m-t - 2} = \Theta\left(\eps^2 n\right).
\end{align*}
Hence,
$\mathbb{E}(X_{\ref{i:coltwo}}) = o\left( \frac{1}{\eps}\right)\Theta\left(\eps^2 n\right) =o(\eps n)$.

\subsection*{Trees not attached to $\hat C$: Case $*=$ \ref{i:notinhatC}.}
Finally note that, by \Cref{rem:expectedtreesize}\eqref{r:expectedtreesize}, $\mathbb{E}_{F,x}(|V(T_x)|) = \frac{m}{t}$, where $x$ is a uniformly chosen random vertex from $[t]$, and so
\[
\mathbb{E}(X_{\ref{i:notinhatC}}) \leq \Big(|V(C)| - |V(\hat{C})|\Big) \frac{m}{t} = o\left(\eps^2 n \right) \frac{(2+o(1))\eps n}{\Theta\left(\eps^2 n\right)}= o(\eps n).
\]

\vspace{0.5cm}

Collecting all of the four cases, we have shown that $\mathbb{E}\left(\sum_{i=1}^4 X_i\right) = \sum_{i=1}^4 \mathbb{E}(X_i) = o(\eps n)$.
In particular, by Markov's inequality, whp
$\sum_{i=1}^4 X_i = o(\eps n)$. 
Hence, whp $\hat{C} \cup \hat{F}$ is a connected, rainbow subgraph of $G_c(n,p)$ of order at least
\[
|L \cup U| - \sum_{i=1}^4 X_i = (1+o(1))2\eps n - o(\eps n) = (1+o(1))2\eps n,
\]
as required.
\end{proof}

\section{Sparse regime}\label{s:sparse}

In this section we consider the regime when $p=d/n$. In the case of the subcritical regime, where $d<1$, as mentioned at the end of \Cref{s:weaklysub}, the proof of \Cref{thm:giantrainbowtree}(\ref{i:sub}) will show that  whp the largest rainbow tree has approximately the same order as the largest component. So, in what follows we will focus on the supercritical regime, where $d>1$.

First, we consider the case where $p$ is only just above the phase transition threshold of $1/n$, and prove \Cref{t:lineartreesparse} by comparing a rainbow breadth-first search exploration process on $G_c(n,p)$ to a branching process, i.e.\ we show that when $c= \alpha n$ and $p = \frac{1+\eps}{n}$ for constant $\alpha$ and for $\eps$ sufficiently small, the largest rainbow tree will have order $\Theta(\eps n)$.
However, given that the likely order of the giant component in $G(n,p)$ is $\left(2\eps + O\left(\eps^2\right)\right)n$, we conjecture that this can be sharpened.

\begin{conjecture}
Let $\alpha> 0$, let $c=\alpha n$, let $\eps >0$ be sufficiently small and let $p = \frac{1+\eps}{n}$. Then whp the largest rainbow tree in $G_c(n,p)$ has order $\left(2\eps + O\left(\eps^2\right)\right)n$.
\end{conjecture}

We show only the weaker lower bound in \Cref{t:lineartreesparse}.

\begin{proof}[Proof of \Cref{t:lineartreesparse}]
We will describe a rainbow breadth-first search (RBFS) exploration process
which, early on its evolution, can be coupled with a binomial branching process.

More precisely, we run a RBFS process on $G_c(n,p)$, starting at an arbitrary root vertex,
where whenever we have more than $(1-\delta)n$ undiscovered vertices in each step,
we arbitrarily forbid some vertices to ensure that precisely $(1-\delta)n$ remain --
this slightly counter-intuitive strategy is for technical convenience later on.
Furthermore, whenever we discover an edge, we only accept the edge if it does not share any colours with the previously accepted edges, and in addition, if the probability
of rejection is less than $\delta/\alpha$, we introduce an additional
edge-rejection probability to ensure that the total probability is always precisely $\delta/\alpha$. 
Again, note that we are counter-intuitively rejecting edges that we could keep,
but this convention will be more convenient.
Each time the queue of active vertices is empty, we choose an arbitrary unexplored vertex as the new root. Observe that this process builds a rainbow forest.

Under the assumption that the forest has not yet grown to order $\delta n$, we see that the probability that a newly discovered edge shares a colour with the current forest is at most $\frac{\delta}{\alpha}$, and for each vertex we will seek neighbours from among
at least $(1- \delta)n$ vertices, and so,
due to the counter-intuitive additional restrictions described above,
each component in this process can be coupled with a Bin$(m,q)$ branching process
where $m=(1-\delta)n$ and $q=(1-\delta/\alpha)p$, so we have
\[
mq = (1-\delta)\left(1-\frac{\delta}{\alpha}\right)(1+\eps).
\]
Let us set $\delta:= \frac{\alpha}{\alpha+1}\eps -\eps^2$.
It is easy to check that $mq > 1 + \Omega\left(\eps^2\right)$, and so this branching process is supercritical.

At this point, standard arguments imply that if we explore the component structure of $G_c(n,p)$ via this RBFS process, whp we will discover a rainbow tree of order at least $\delta n - \eps^2 n= \left(\frac{\alpha}{\alpha+1}\eps - O\left(\eps^2\right)\right)n$.

More precisely, we run the RBFS process until one of the following two stopping conditions
is reached.
\begin{enumerate}[(S1)]
    \item \label{stop:large} We have discovered a rainbow tree of order at least $\delta n - \eps^2 n$.
    \item \label{stop:time} The queue of active vertices is empty and the total size of the
    forest is at least $\eps^2 n$.
\end{enumerate}
Note that the coupling is indeed valid as long as neither of these conditions
has been applied. We claim that whp Condition~(S\ref{stop:time}) will not be invoked.

To see why this is true, observe that each time we start a new tree, it has probability
$\Omega\left(\eps^2\right)$ of reaching order $\delta n - \eps^2 n$
and therefore triggering Condition~(S\ref{stop:large}) (roughly corresponding to the survival probability of the corresponding branching process),
independently for each tree.

Furthermore, conditioned on a tree \emph{not} reaching order
$\delta n- \eps^2 n$
(and in particular certainly not surviving forever)
the RBFS while discovering this tree behaves approximately
like a \emph{subcritical} branching process with growth rate
$1-\Omega\left(\eps^2\right)$.
The expected order of such a branching process is $O\left(\eps^{-2}\right)$,
and so for any $\omega\to \infty$, the (still conditional)
probability that the first $\omega$ such processes have total order $\omega^2 \eps^{-2}$ is $o(1)$ by Markov's inequality.
Choosing $\omega$ to grow slowly enough that $\omega^2 \eps^{-2} \le \eps^2 n$,
we deduce that whp we would have to start at least $\omega$ processes in order
for Condition~(S\ref{stop:time}) to be invoked. But since $\omega\to \infty$
and each process has probability $\Omega\left(\eps^2\right)$ of growing large,
whp one of the first $\omega$ processes will indeed grow large.
\end{proof}

In order to prove \Cref{t:almostspanningcycle}, we will first show the likely existence of an almost spanning path, and then complete the proof via a standard sprinkling argument. In order to show the existence of a long rainbow path we will use a rainbow depth-first search process, following a technique of Krivelevich and Sudakov \cite{KS13}.

\begin{theorem}\label{t:almostspanningpath}
Let $\delta,\alpha >0$ and let $c=\alpha n$, let $d$ be such that $d \geq \frac{16}{\delta^3}$, and let $p=\frac{d}{n}$. Then whp $G_c\left(n,p\right)$ contains a rainbow path of length $(1-\delta)\min\{ c,n\}$.
\end{theorem}
\begin{proof}
We consider the following rainbow depth-first search (RDFS) process on a graph: We maintain a stack $A$ of \emph{active} vertices as well as collections $W$ and $U$ of \emph{visited} and \emph{unvisited} vertices. At each step of the process $A=(v_1,v_2,\ldots, v_k)$ will span a rainbow path $P_{v_k}$ in $G_c(n,p)$.

In a step, we consider the vertex $v$ at the top of the stack and
for each previously unqueried pair $vu$ with $u\in U$ in some arbitrary order
we \emph{query} the pair to discover whether it forms an edge in $G(n,p)$, and if so what its colour is.
We accept the query if the pair forms an edge whose colour does
not appear in $P_{v}$ -- in this case the vertex $u$
is moved from $U$ to the top of the stack $A$ and we continue querying
from $u$ -- otherwise the query is rejected and we continue querying from $v$.
If all pairs $vu$ have been queried, we move $v$ from $A$ to $W$. If the stack becomes empty, an arbitrary vertex from $U$ is moved into $A$.

Note that, for any fixed vertex $v$, each time $v$ is the top vertex in the stack, the path $P_v$ will be the same. Let us write $L_v$ for the set of colours used in $P_v$.
Note that, for any vertex $w \in W$, every edge from $w$ to $U$ has been exposed, and those that are in $G(n,p)$ must have a colour in $L_w$.

By the principle of deferred decisions, we may
alternatively describe this RDFS process by first picking an independent sequence of pairs of random variables $(X_i,Y_i)_{i \in \mathbb{N}}$ where $X_i \sim \text{Ber}(p)$ and $Y_i$ is uniformly distributed amongst the set of $c$ colours independently of $X_i$,
and when we run our algorithm, when we query the $i$-th pair we reject the query if and only if $X_i = 0$ or $Y_i \in L_v$ (where $v$ is the top vertex in the stack at this point).

Let us fix $r = \min \{c,n\}$ and suppose we run our process until $N = \frac{\delta^2 rn}{8}$ pairs have been queried and that at no point in this process was $|A| \geq (1-\delta)r$. We first claim that at this point there are at most $\frac{\delta r}{2}$ vertices in~$W$. Indeed, if not, there is some point after $N' \leq N$ queries at which $|W| = \frac{\delta r}{2}$ and hence, since $|A| \leq (1-\delta)r$ by assumption, at this point $|U| \geq \frac{\delta n}{2}$, since $|W| + |U| + |A| = n \geq r$. However, in this case we must have already queried at least $|U||W| = \frac{\delta^2 r n}{4} > N$ many pairs, a contradiction.

Hence, we may assume that $|W| \leq \frac{\delta r}{2}$, and so in particular, since $|A| \leq (1-\delta)r$ and $r \leq n$ we are still in the process of exploring the graph, and each query that was not rejected moved a vertex from $U$ to $A$. However, by assumption $|L_v| = |P_v| \leq (1-\delta)r$ for each $v \in W$. Hence, the probability that each query is accepted is at least $p\left( 1-\frac{ (1-\delta)r}{c}\right) \geq \delta p$, and so by the Chernoff bound (\Cref{lem:chernoff}), whp the number of queries in the first $N$ which are accepted is at least
\[
\frac{1}{2} N \delta p = \frac{d \delta^3r}{16} \geq r,
\]
since $d \geq \frac{16}{\delta^3}$. In particular, $|A \cup W| \geq r$, and so $|A| \geq r - |W| \ge \left(1- \frac{\delta}{2}\right)r \geq (1-\delta)r$, a contradiction.
\end{proof}

It is easy to conclude the existence of long rainbow cycles in this model from the existence of long rainbow paths via a simple sprinkling argument.

\begin{proof}[Proof of \Cref{t:almostspanningcycle}]
Let $d$ be such that $d-1 \geq \frac{128}{\delta^3}$. Let $p = \frac{d}{n}$, let $p_1 = \frac{d-1}{n}$ and let $(1-p_2)(1-p_1) = (1-p)$, noting that $p_2 \geq \frac{1}{n}$. 

We generate $G(n,p)$ by taking two independent random graphs $G_1\sim G(n,p_1)$ and $G_2\sim G(n,p_2)$ and letting $G = G_1 \cup G_2$, so that $G \sim G(n,p)$. If we then uniformly colour the edges of $G$ from a set of $c$ colours, then it is clear that the resulting coloured graph is distributed as $G_c(n,p)$ and furthermore that the induced colouring on the graph $G_1$ is distributed as $G_c(n,p_1)$.

Letting $r = \min\{c,n\}$, then by \Cref{t:almostspanningpath}
with $\delta/2$ in place of $\delta$ and $d-1$ in place of $d$,
whp $G_1$ contains a
rainbow path $P$ of length $\left(1-\frac{\delta}{2}\right)r$. At this point we aim to use a sprinkling argument to close $P$ to a rainbow cycle of
almost the same length by finding an edge of $G_2$
between vertices close to the ends of $P$. However, we need to be careful
since if this is also an edge of $G_1$, then we may
already have revealed its colour. We therefore first show that
there are not too many such edges.

To this end, we note that for any pair of disjoint sets $X$ and $Y$ of size 
$\frac{\delta}{4}n$ the expected number of edges between $X$ and $Y$ in $G_1$
is $\frac{\delta^2(d-1)}{16}n$ and so, by the Chernoff bound (\Cref{lem:chernoff}),
with probability $\exp \left( - \Omega\left(n^2\right) \right)$
there are at most $n^{\frac{3}{2}}$ edges between $X$ and $Y$, 
and so by a union bound whp there are at most
$n^{\frac{3}{2}}$ edges of $G_1$ between any two such sets $X,Y$.

In particular, if we let $X$ and $Y$ be the first and last $\frac{\delta}{4}r \le \frac{\delta}{4}n$ vertices on $P$, then whp there are at most $n^{\frac{3}{2}}$ edges of $G_1$ between $X$ and $Y$.
Hence, when we expose $G_2$, it is again a simple consequence of the Chernoff bound (\Cref{lem:chernoff}) that whp there are at least
\[
\frac{1}{2} \left( |X||Y| - n^{\frac{3}{2}} \right) p_2 \geq \frac{ \delta^2 r^2}{64n}
\]
edges between $X$ and $Y$ in $G_2$ which are not in $G_1$. Furthermore, for each of these edges the probability that it uses a colour not in $P$ is at least $\frac{\delta}{2}$ and so
with probability at least
\[
1-\left(1-\frac{\delta}{2}\right)^{\frac{ \delta^2 r^2}{64n}}
\ge 1-\exp\left(-\frac{\delta^3 r^2}{128n}\right) = 1-o(1),
\]
there is at least one edge $e$ in $G_2$ between $X$ and $Y$ which uses a colour not used in $P$.
It follows that there is a rainbow cycle $C \subseteq P+e \subseteq G$ of length at least $\left(1- \frac{\delta}{2} - 2 \frac{\delta}{4}\right)r = (1-\delta)r$.
\end{proof}

A similar argument shows the existence of a linear rainbow cycle when $p = \frac{d}{n}$ for arbitrary $d >1$.

\begin{theorem}\label{t:longpath}
Let $\alpha >0$, let $c=\alpha n$, let $\eps >0$ be sufficiently small and let $p= \frac{1+2\eps}{n}$. Then whp $G_c(n,p)$ contains a rainbow cycle of length $\Omega\left( \eps^2 n\right)$.
\end{theorem}
\begin{proof}
We first show that whp $G_c\left(n,\frac{1+\eps}{n}\right)$ contains a rainbow path of length $\Omega(\eps^2 n)$.
To do this we consider, as in the proof of \Cref{t:almostspanningpath}, the RDFS process on $G_c(n,p_1)$, where $p_1 = \frac{1+\eps}{n}$.
Let us set $\delta := \frac{\eps^2 n}{5c} = \Theta(\eps^2)$.
Suppose that we run the process until $N = \frac{\eps}{2} n^2$ edges have been queried and that at no point in this process was $|A| \geq \frac{\eps^2}{5} n = \delta c$, where recall that $A$, the set of active vertices, forms a path.

We first claim that after $N$ queries there are at most $\frac{n}{3}$ vertices in $W$, the set of visited vertices. Indeed, if not, there is some point after $N' \leq N$ queries where $|W| = \frac{n}{3}$ and hence, since $|A| \leq \frac{n}{3}$, at this point the set $U$ of unvisited vertices satisfies $|U| \geq \frac{n}{3}$. However, in this case we must have already queried at least $|U||W| = \frac{n^2}{9} > N$ many edges, a contradiction.

Hence, we may assume that $|W| \leq \frac{n}{3}$, and so in particular, since $|A| \leq \frac{n}{3}$, we are still in the process of exploring the graph, and each query that was not rejected moved a vertex from $U$ to $A$. However, by assumption $|L_v| = |P_v| \leq \delta c$ for each $v \in W$. 
Hence, the probability that any query is accepted is at least $(1-\delta) p_1$, and so by the Chernoff bound (\Cref{lem:chernoff}) whp the number of accepted queries in the first $N$ is at least
\[
(1-\delta)p_1 N - \left((1-\delta)p_1 N\right)^{2/3} \geq \left(\frac{\eps}{2} + \frac{\eps^2}{2} - O\left(\eps^3\right)\right)n.
\]
Hence, at this point whp $|W| \geq |A \cup W| - |A| \geq \left(\frac{\eps}{2} + \frac{3\eps^2}{10}- O\left(\eps^3\right)\right)n$.
Furthermore, at this point $|U| \geq n - |W| - \frac{\eps^2 n}{5} $,
and $\left(n - |W| - \frac{\eps^2 n}{5} \right)|W|$ is minimised by taking $|W|$ as small as possible because $|W| \leq \frac{n}{3}$.
It follows that whp we have queried at least 
\begin{align*}
|U||W| & \ge \left(n - |W| - \frac{\eps^2 n}{5} \right)|W|
\geq  \left(1 - \frac{\eps}{2} - \frac{\eps^2}{2} - \frac{\eps^2}{5} +O\left(\eps^3\right)\right)
\left(\frac{\eps}{2} + \frac{3\eps^2}{10} - O\left(\eps^3\right)\right)n^2\\ 
&= \left( \frac{\eps}{2} + \eps^2 \left(\frac{3}{10} - \frac{1}{4} \right) + O\left(\eps^3\right) \right) n^2 \geq \left( \frac{\eps}{2} + \frac{\eps^2}{20} + O\left(\eps^3\right) \right) n^2 > N
\end{align*}
pairs, a contradiction. Hence, whp $G_c(n,p_1)$ contains a rainbow path $P$ of length $\delta n$.

Finally we can close this path to a rainbow cycle of almost the same length
by a sprinkling argument essentially identical to the one in the proof of \Cref{t:almostspanningcycle}. We omit the details.
\end{proof}

\section{Discussion}\label{s:discussion}
While in \Cref{thm:giantrainbowtree} we have determined the asymptotic order of the largest rainbow tree in $G_c(n,p)$ in the weakly supercritical regime, i.e.\ when $\eps \to 0$,
there is still a large gap between the upper and lower bounds in the supercritical regime, where $\eps >0$ is a small constant. When $\alpha = \frac{c}{n}$ is small compared to $\eps$, it is clear that the naive upper bound given by the order of the giant component cannot be asymptotically optimal, since there will be too few colours
to find a rainbow almost spanning subtree of the giant.
This is clear when $\alpha \ll \gamma(1+\eps)$, but in fact this will be an issue even for larger values of $\alpha$.

Indeed, it is known, see e.g.~\cite{FK16b}, that whp for this regime of $p$,
setting $1+\eps=d$ and $\gamma=\gamma(d)$,
the largest component has approximately $\gamma n$ vertices 
and $\left(1 - (1-\gamma)^2\right) \frac{dn}{2}$ edges.
However, since $\gamma$ satisfies the equation $1-\gamma = e^{-\gamma d}$, we see that the probability $q(\alpha)$ that any particular colour $i \in [c] = [\alpha n]$ is contained in the giant component is approximately
\begin{align*}
q(\alpha) &= 1- \left(1 - \frac{1}{c}\right)^{(1 - (1-\gamma)^2) \frac{dn}{2}}\\
&\approx 1 - \exp\left(-\frac{\gamma d}{\alpha} \left(1 - \frac{\gamma}{2}\right)\right)\\
&= 1 - (1-\gamma)^{\frac{1 - \frac{\gamma}{2}}{\alpha}}.
\end{align*}
In particular, if $\alpha \leq 1$, then, approximately, the expected number of colours appearing in the giant component is 
\[
\alpha q(\alpha)n \leq q(1)n = \left(1 - (1-\gamma)^{1 - \frac{\gamma}{2}}\right)n < (1-(1-\gamma)^{1})n = \gamma n, 
\]
and so we do not expect there to be sufficiently many colours appearing in the giant component to be able to cover almost all of the vertices with a rainbow tree.

However, even in the case where $\eps \ll 1 \ll \alpha$, it is not clear if the largest rainbow tree covers most of the giant component. 

\begin{question}
For fixed $\alpha  >0$ and $d >1$ what is the order of the largest rainbow tree in $G_{\alpha n}\left(n, \frac{d}{n}\right)$? 

In particular, for small $\alpha$ is this asymptotically the same as the number of different colours appearing in the giant component? Conversely, if $\alpha$ is large enough, is this asymptotically as large as the giant component?
\end{question}

In this regime, parts of the argument in \Cref{s:weakly} will fail to hold whp, even for $\eps$ arbitrarily small, as the bounds on some probabilities are only tending to $0$ as $\eps\to 0$. However, it may be the case that a careful analysis of the arguments in this regime will still show that the statement holds with positive probability. In this case, it is conceivable that the existence of a large rainbow tree could be deduced by showing that the order of the largest rainbow tree is well-concentrated about its mean.

\begin{question}
Let $\alpha,\eps > 0$ be constant, let $c = \alpha n$, let $p=\frac{1+\eps}{n}$ and let $T$ be the order of the largest rainbow tree in $G_c(n,p)$. Is it true that whp $T = (1+o(1))\mathbb{E}(T)$?
\end{question}

\Cref{t:longpath} raises the natural question of how large the longest rainbow path and cycle are in $G_c(n,p)$ in the weakly supercritical regime. In fact, a careful analysis of the proof of \Cref{t:longpath} will show that it holds for certain ranges of $\eps$ in the weakly supercritical regime, although the sprinkling argument seems to require that $\eps^5 n \to \infty$. Note that there is a corresponding upper bound of $O(\eps^2 n)$, given by known bounds on the length of the longest path and cycle in $G(n,p)$ in this regime, see \cite{L91}.

It is likely that arguments as in \Cref{s:weakly} can prove a corresponding lower bound all the way to the critical window.
Namely, we have seen that we can delete a set of 
$o\left(\eps^3n\right)$ edges from the $2$-core of $G(n,p)$
in order to produce a rainbow subgraph $H$.
Arguments as in \cite[Section 4]{L91} (see also \cite[Lemma 4.3]{DEK22}) 
should show that the the kernel $K(H)$ of $H$ contains an almost spanning 
subgraph which is distributed as a random $3$-regular graph, which is known to have good expansion properties whp (see, for example, \cite{B88}). In particular, standard results on graph expansion (see, for example, \cite{K19}) will then imply that whp $K(H)$ contains a cycle $C$ of length $\Theta(|K(H)|)$, and arguments as in \cite[Theorem 6]{L91} will show that whp the corresponding cycle in $H$, given by subdividing the edges in $K(H)$ appropriately, will have length $\Theta(|H|) = \Theta\left(\eps^2 n\right)$.

As well as the length of the longest cycle, there has also been much interest in the \emph{cycle spectrum} $\mathcal{L}(n,p)$ of $G(n,p)$, the set of lengths of cycles. Cooper and Frieze \cite{CF90} showed that above the Hamiltonicity threshold whp $G(n,p)$ is \emph{pancyclic} --
it contains cycles of all possible lengths, i.e. $\mathcal{L}(n,p) = [3,n]$. {\L}uczak \cite{L91a} showed that if $np \to \infty$ then whp $\mathcal{L}(n,p) \supseteq [3,n-(1+\eps)N_1]$ where $N_1$ is the number of isolated vertices in $G(n,p)$.
More recently Alon, and Krivelevich and Lubetzky \cite{AKL21}
showed that in the sparse regime, for any $\omega \to \infty$,
whp $\mathcal{L}(n,p) \supseteq [\omega,(1-\eps)L]$,
where $L$ is the length of the longest cycle, and this was improved to $\mathcal{L}(n,p) \supseteq [\omega,L]$ by Anastos \cite{A21}.
It would be interesting to know if similar statements hold for the \emph{rainbow cycle spectrum} in $G_c(n,p)$, for appropriate ranges of $p$ and $c$.

Finally, it would have been convenient in some of the proofs in \Cref{s:sparse} to assume that the size of various rainbow parameters are increasing with the number of colours used.
Indeed, it seems intuitively obvious that if $X(c)$ is a random variable which counts the number of a particular rainbow substructure in $G_c(n,p)$, then $X(c)$ should be at least as large as $X(c')$ for all $c' \leq c$, since using more colours should only make things `more rainbow'.
However, while when $c=2c'$ it is easy to demonstrate a coupling between $G_c(n,p)$ and $G_{c'}(n,p)$ which preserves the property that a subgraph is rainbow,
in general we were not able to find such a coupling. However, we believe that such a `monotonicity' statement should be true, and in particular we make the following conjecture.

\begin{conjecture}
Let $k,c,c' \in \mathbb{N}$ be such that $c' \leq c$ and let $G$ be a graph.
Define $T(c,k)$ to be the random variable which counts the number of rainbow trees of order $k$ when we uniformly colour the edges of $G$ with $c$, and define $T(c',k)$ similarly. Then $T(c,k)$ stochastically dominates $T(c',k)$.
\end{conjecture}

\bibliographystyle{plain}
\bibliography{rainbow}

\end{document}